\documentclass[a4paper,11pt]{article}
\usepackage{aligned-overset}
\usepackage{graphicx}        
\usepackage{multicol}        
\usepackage[bottom]{footmisc} 
\usepackage{amsmath}

\usepackage{epsfig}
\usepackage{authblk}
\usepackage{float}
\usepackage{diagbox}
\usepackage{nicefrac}
\usepackage[T1]{fontenc}
\usepackage{blkarray}
\usepackage{mathtools,amsthm,amsfonts}   
\usepackage{citesort}
\usepackage[style=numeric-comp,maxbibnames=99,bibencoding=utf8,firstinits=true]{biblatex}
\bibliography{elasticity_robust_nodal.bib}
\usepackage{helvet}  

\usepackage{newtxtext}
\usepackage[upint,varg]{newtxmath}

\usepackage{hhline}
\usepackage{tikz}

\usepackage{color}

\usepackage{bm}
\usepackage{nicefrac}
\usepackage[hmargin={28mm,28mm},vmargin={30mm,35mm}]{geometry}
\usepackage{array}
\usepackage{pdfpages}
\usepackage{subfig}
\usepackage{enumitem}
\usepackage{bbold}
\usepackage{xparse}								
\usepackage{stmaryrd}							
\usepackage{sidecap}								
\usepackage{upgreek}
\usepackage{accents}
\allowdisplaybreaks
\usepackage{tcolorbox}
\usepackage{algorithmic}
\usepackage[ocgcolorlinks,allcolors={blue}]{hyperref}
\usetikzlibrary{shapes,arrows,positioning,calc}
\usepackage{alltt}
\usepackage[bbgreekl]{mathbbol}
\DeclareSymbolFontAlphabet{\mathbb}{AMSb}
\DeclareSymbolFontAlphabet{\mathbbl}{bbold}
\def\N{\mathbb N}
\def\R{\mathbb R}

\def\U{\mathbf{U}}

\def\D{\mathcal{D}}

\def\<{\langle}
\def\>{\rangle}

\def\lsim{\lesssim}

\def\wt{\widetilde}
\def\Chi{\raise .3ex \hbox{\large $\chi$}} 

\def\ov{\overline}

\def\normal{\mathbf{n}}

\def\la{\lambda}

\newcommand{\IPbracket}[2]{\langle #1,#2\rangle}

\def\dsp{\displaystyle}
\def\x{{\bf x}}

\def\U{{\bf U}}
\def\p{{\bf p}}

\def\cells{{\mathcal{M}}}
\def\faces{{\mathcal{F}}}
\def\nodes{{\mathcal{V}}}
\def\edges{{\mathcal{E}}}

\def\cF{\mathcal{F}}

\def\x{{\bf x}}
\def\dsp{{\displaystyle x}}

\def\dsp{\displaystyle}
\def\bu{\mathbf{u}}

\def\bv{\mathbf{v}}
\def\bw{\mathbf{w}}

\newcommand{\bs}[1]{\mathbf{#1}}

\newcommand{\UD}{\mathbf{U}_\D}
\newcommand{\UDz}{\mathbf{U}_{0,\D}}

\newcommand{\Poly}[1]{\mathbb{P}^{#1}}
\newcommand{\Polygras}[1]{\mathbf{\mathbb{P}}^{#1}}

\newcommand{\Sigmaspace}{\mathbf{\Sigma}}    

\newcommand{\IUD}{\mathcal I_{\UDz}} 

\newcommand{\NORM}[2]{\|#2\|_{#1}}
\newcommand{\SEMINORM}[2]{|#2|_{#1}}

\newcommand{\Ks}{{Ks}}
\def\Ksig{{K\!\sigma}}

\newcommand{\sige}{{\sigma e}}

\def\bbsig{\bbsigma}
\def\bbeps{\bbespilon}

\newcommand{\jump}[1]{\llbracket #1 \rrbracket}

\newcommand{\email}[1]{\href{mailto:#1}{#1}}

\begingroup
    \makeatletter
    \@for\theoremstyle:=definition,remark,plain\do{%
        \expandafter\g@addto@macro\csname th@\theoremstyle\endcsname{%
            \addtolength\thm@preskip\parskip
            }%
        }
\endgroup

\newtheorem{theorem}{Theorem}
\numberwithin{theorem}{section}

\newtheorem{lemma}[theorem]{Lemma}

\theoremstyle{remark}
\newtheorem{remark}[theorem]{Remark}
\theoremstyle{definition}

\newtheorem{definition}[theorem]{Definition}
\newtheorem{example}[theorem]{Example}

\def\grad{\nabla}

\def\div{{\rm \textbf{div}}}
\def\divScal{{\rm div}}

\definecolor{labelkey}{rgb}{0.6,0,1}

\usepackage{marginnote}

\usepackage{parskip}
\usepackage{algorithmic}

\usepackage{diagbox}
\usepackage{booktabs}

\makeatletter
\def\thm@space@setup{%
  \thm@preskip=\parskip \thm@postskip=0pt
}
\makeatother

\usepackage[normalem]{ulem}
\newcounter{corr}
\definecolor{violet}{rgb}{0.580,0.,0.827}
\newcommand{\corr}[3]{\typeout{Warning : a correction remains in page \thepage}
  \stepcounter{corr}
 	      {\color{blue}\ifmmode\text{\,\sout{\ensuremath{#1}}\,}\else\sout{#1}\fi}
              {\color{red}#2}
              {\color{violet} #3}
}

\begin{document}
\title{A locking-free nodal-based polytopal method for linear elasticity}

\author[1,2]{{J\'er\^ome Droniou}\footnote{\email{jerome.droniou@cnrs.fr}}}
\author[1]{{Raman Kumar}\footnote{ \email{raman.kumar@umontpellier.fr}}}
\affil[1]{IMAG, Univ. Montpellier, CNRS, Montpellier, France.}%
\affil[2]{School of Mathematics, Monash University, Australia}%

\date{}
\maketitle
\vspace{-.5cm}
{\footnotesize{
\centering
}}

\begin{abstract}
This work presents a Discrete de Rham (DDR) numerical scheme for solving linear elasticity problems on general polyhedral meshes, with a focus on preventing volumetric locking in the quasi-incompressible regime. The method is formulated as a nodal-based approach using the lowest-order gradient space of the DDR complex, enriched with scalar face bubble degrees of freedom that effectively capture the normal flux across element faces. This face-bubble enrichment is crucial for ensuring sufficient approximation flexibility of the divergence field, thereby eliminating the {volumetric locking} phenomenon that typically occurs as the Lam\'e parameter $\la$ approaches infinity. We establish $H^1$-error estimates that are independent of $\la\ge 0$, and depend only on the lower bound of $\mu$, guaranteeing robustness across the entire range from compressible to nearly incompressible regimes. We also show how to adapt our scheme to the frictionless contact mechanics model, maintaining a locking-free   estimate for the primal variable~(displacement). Numerical experiments confirm that the proposed {locking-free} method delivers accurate and stable approximations on general polytopal discretizations, even when the material behaves as an incompressible medium. The flexibility and robustness of this approach make it a practical alternative to mixed formulations for engineering applications involving nearly incompressible elastic materials.
  
\end{abstract}

%

\section{Introduction}
Volumetric locking is a phenomenon associated with numerical techniques for elasticity models, and describe the ill-behaviour of the scheme when the second Lamé coefficient $\la$ becomes large. For $\la \to+\infty$, that is, when the elastic material is nearly incompressible, it is well known that the lowest-order nodal-based finite element methods provide poor convergence rates for the displacement, see \cite{MR1174468}. The underlying reason is that the corresponding discrete space is not reach enough in divergence-free functions, the limit $\la\to+\infty$ formally imposing a zero divergence on the displacement.

To overcome the effects of the {volumetric locking}, several numerical approaches have been designed over the years. One approach is to apply the mixed finite element methods~(FEMs), where the model is reformulated into a model that involves both displacement and pressure variables. The advantage of these mixed techniques is that they provide a good convergence rate for both displacement and pressure, see \cite{brezzi2012mixed}; however, a well-known limitation of mixed finite element discretisations is the need to satisfy the inf-sup condition. In our present work, we focus exclusively on {Poisson locking} in situations where the primary concern is the accuracy of the displacement, and not that of the pressure variable. Our main objective is to design a {robust} nodal numerical method with scalar {bubble} enrichment formulated in the primal~(displacement) form, as opposed to the mixed formulation. By {robust}, we mean a method that performs well and provides accurate results over a practical range of mesh discretizations, remaining stable and effective even when the material parameter approaches its limiting value. One of the interests of considering a primal form is that it naturally provides displacements that are more regular than in mixed formulations (e.g., they are continuous at the vertices). Moreover, we also aim at designing a scheme that is applicable to more generic meshes than classical finite elements.

In recent years, a variety of polytopal methods -- capable of operating on general polygonal and polyhedral meshes -- have been developed, including Discontinuous Galerkin Method \cite{hansbo-larson}, Discrete de Rham (DDR) method~\cite{ddr-variant}, Hybrid High Order (HHO) Method~\cite{hho-book}, Hybrid Mimetic Mixed Methods \cite{dipietro-lemaire}, and Virtual Element Method (VEM)~\cite{beirao-brezzi-marini, da2015virtual}. These approaches have rapidly gained considerable attention within the numerical analysis community owing to their flexibility in handling complex geometries and their ability to provide accurate approximations for problems involving high-regularity solutions.

For a two-dimensional linear elasticity problem, a lowest-order robust nodal VEM has been designed, see \cite{Tang-2020, Huang-2023}. In these references, the robustness is achieved by adding vector-valued degrees of freedom~(DOFs) at the edge midpoints, representing full displacements. A similar approach is adopted in \cite{Huang-2025} for the 3-dimensional elasticity model, where an additional vector-valued DOF is added at a point in each face, and the considered virtual functions are piecewise linear on a sub-triangulation based on that point. It is worth noting that the additional nodal displacement unknowns are added to enrich the space of divergences of the virtual functions (thus achieving robustness); adding a full displacement unknown to only enrich the divergence, however, does not appear to be the most cost-effective choice. Recently, {volumetric locking-free} mixed VEMs for two-dimensional frictionless contact problems are discussed \cite{lovadina2026volumetric}, employing a mixed displacement/pressure formulation to prevent {volumetric locking}. Their analysis yields explicit error estimates suitable for meshes with small edges, and numerical experiments on both first- and second-order schemes confirm theoretical predictions and robustness properties.

In this work, we design a DDR numerical scheme for both two and three-dimensional versions of the linear elasticity problem. We start from a nodal (vertex-based) formulation based on the first component of the DDR complex and propose to enrich the discrete space, together with the associated operators, by introducing additional {scalar} face-bubble unknowns. These unknowns are designed to represent normal fluxes across faces and thus provide a targeted enriched approximation of the divergence of the displacement field, without having to enrich the whole displacement on the face. This additional unknown is crucial to ensure that the discrete space remains sufficiently flexible in the quasi-incompressible limit, thereby preventing {volumetric locking}. An {optimal} order error estimate is derived for an $H^1$-like norm, the key feature being its independence on the second Lam\'e coefficient. It should be noted that,  polytopal or not,    {face}-based schemes are often naturally {locking-free} \cite{dipietro-ern,dipietro-lemaire}, the drawback being that they provide approximate solutions that are less regular than nodal-based methods \cite{crouzeix1973conforming}.  Moreover, we also show that the proposed {locking-free} approach extends to contact mechanics with Tresca friction on fractured media. The {bubble} enrichment strategy naturally combines with mixed formulations for the frictionless contact model designed in \cite{MR4886366}, maintaining the robustness properties in the quasi-incompressible limit for the primal variable~(displacement).

{Layout of the article:} In Section~\ref{sec:scheme}, we start with the formulation of the linear elasticity model~(Subsection~\ref{sec:model}), then present the proposed DDR scheme, detailing the polytopal mesh structure and the discrete spaces with nodal and scalar face-bubble degrees of freedom~(Subsection~\ref{subsec:mesh}), the gradient and displacement reconstruction operators on faces and cells~(Subsection~\ref{subsec:operators}), the interpolator (Subsection~\ref{subsec:interpolator}), and the discrete variational formulation with stabilization (Subsection~\ref{subsec:mixed}). The main theoretical result is presented in Section~\ref{sec:results}, namely, an error estimate in a discrete $H^1$-norm with a constant that does not depend on the second Lamé coefficient, which guarantees the robustness in the incompressible limit. Section~\ref{sec:proofs} contains the supporting analysis, including a discrete Korn inequality (Theorem~\ref{th:korn}), commutation property (Lemma~\ref{lm:com}), and detailed derivations of the abstract error and consistency estimates (Subsection~\ref{sec:proof.abstract.error}). Section~\ref{Tresca_friction_extension}  extends the method to frictionless contact mechanics model, confirming that the {locking-free} properties extend to contact problems for the primal variable~(displacement). Moreover, extension to Tresca contact is briefly covered in Remark~\ref{rem:tresca}. Finally, in Section~\ref{sec:numerics}, the simulations were done with the HArDCore3D code, and linked to the public repository: https://github.com/jdroniou/HArDCore3D-release. These numerical experiments confirm the accuracy, robustness, and {locking-free} performance of the proposed method on general polyhedral meshes. 
\section{Model and scheme}\label{sec:scheme}

This section introduces the Discrete de Rham (DDR) numerical scheme for linear elasticity on polyhedral meshes, including the continuous model formulation, discrete spaces with face-bubble enrichment, reconstruction operators, the interpolator, and the discrete variational formulation with stabilization.

\subsection{Model}\label{sec:model}
Let $\Omega \subset \mathbb{R}^{d}$, $d\in \{2,3\}$, denote a bounded connected polytopal domain with boundary $\partial\Omega$. We consider the linear elasticity problem which consists in searching for the displacement $\bu:\Omega\to\R^d$ such that
\begin{equation}
\label{eq:model.strong} 
\left\{\!\!\!\!
\begin{array}{lll}
& -\div \bbsigma(\bu)= \mathbf{f}  & \mbox{ on } \Omega,\\[1ex]
& \bbsigma(\bu)=2\mu\bbeps(\bu)+\la (\divScal\, \bu )\mathbb{I} & \mbox{ on }  \Omega,\\[1ex]
&\bu={\bf 0} & \mbox{ on } \partial\Omega,
\end{array}
\right.
\end{equation}
where $\mu$ and $\la$ are the Lam\'e coefficients satisfying $\mu \in [\mu_{1},\mu_{2}]$ with $0<\mu_{1}\leq\mu_{2}$ and $\la \in (0,\infty)$, $\bbeps(\bu):= \frac{1}{2}(\nabla\bu+\prescript{t}{}{\nabla\bu})$ is the symmetric gradient operator, $\mathbb{I}$ is the $d\times d$ identity matrix, and $\mathbf{f}\in\mathbf{L}^2(\Omega)=L^{2}(\Omega)^{d}$ represents the external forces. 

The weak formulation of problem \eqref{eq:model.strong} consists in finding displacement $\bu \in \U_0:=\{{\bf v}\in \mathbf{H}^{1}(\Omega):\; {\bf v}={\bf 0}\; \text{on}\; \partial \Omega\}$ such that, for all $\bv \in \U_0$, 
\begin{equation}\label{Lagrange_meca_contactfriction_1}
		 \dsp \int_\Omega \bbsig( \bu): \bbeps( \bv)  
		\dsp =  \int_\Omega \mathbf{f}\cdot \bv.
\end{equation}

For $X \subset \bar{\Omega}$, we denote by $(\cdot, \cdot)_{X}$ and $\NORM{X}{\cdot}$ the standard inner product and norm in  $L^2(X)$, respectively.
An analogous notation is used for $\mathbf{L}^2(X)$ and $L^{2}(X)^{d\times d}$. 

\subsection{Mesh and spaces}\label{subsec:mesh}

We consider a polytopal mesh as in \cite{ddr-variant}.
The set of cells $K$, the set of faces $\sigma$,  the set of nodes $s$, and the set of edges $e$ are denoted respectively by  $\cells$, $\faces$, $\nodes$, and $\edges$. Let $ \nodes^{\text{ext}}$ (resp.\ $ \faces^{\text{ext}}$) and $\nodes^{\text{int}}$ (resp.\ $\faces^{\text{int}}$) be respectively the sets of exterior and interior nodes (resp. faces) of $\nodes$ (resp.\ $\faces$) with respect to the domain $\Omega$. For each $K\in\cells$, $\faces_K$ denotes the set of faces of $K$ and, if $\sigma\in\faces_K$, $\normal_{\Ksig}$ is the outer unit normal to $K$ on $\sigma$. The set of nodes of $X\in\cells\cup\faces$ is denoted by $\nodes_X$, and the set of (one or two) cells adjacent to a given face $\sigma$ is $\cells_{\sigma}$. The coordinate vector of a node $s\in\nodes$ is denoted by $\x_s$. For each cell $K$ and face $\sigma$, let $h_K$ be the cell diameter, and $|K|$ and $|\sigma|$ their respective measures. The mesh size is defined by
$h=\max_{K\in \cells}h_{K}$.

Each $\sigma\in\faces$ is endowed with a fixed unit normal $\normal_{\sigma}$, which determines its orientation, and we denote by $w_{\Ksig}=\pm 1$ the relative orientation of $\sigma$ with respect to $K\in\cells_\sigma$; this is the number such that $w_{\Ksig} \normal_{\Ksig}=\normal_{\sigma}$.
For each $\sigma \in \faces$ we denote by $\gamma^{\sigma}$ the trace operator on $\sigma$ for functions in $H^1(K)$ (or their vector-valued versions). 

Throughout this paper, we suppose that the mesh regularity assumptions of \cite[Definition 1.9]{hho-book} hold, and we write $a \lsim b$ to denote $a \le C b$, where $C > 0$ depends only on $\Omega$, on the mesh regularity parameter, and $\mathbf{f}$, but is independent of the Lam\'e coefficients $\mu$ and $\lambda$. 

If $X\in\cells\cup\faces$ and $\ell\in\N$, we denote by $\Poly{\ell}(X)$ the space of polynomials of total degree $\le \ell$ on $X$. We use the notation $\Poly{\ell}(\cells)$ for the space of piecewise-polynomials of degree $\le \ell$ on $\cells$.

The discrete space is built on the (vector-valued) lowest-order gradient space of the Discrete De Rham (DDR) method \cite{ddr-variant}, which will give us access, in the sections below, to the corresponding gradient and displacement reconstructions. To ensure the robustness in the quasi-incompressible limit, we, however, add to this space face unknowns:
\begin{equation*}
	\begin{aligned}\UD = \Big\{\bv_\D={}&((\bv_{s})_{s\in\nodes},(v_{\sigma})_{\,\sigma\in \faces})\,: \bv_s\in\R^d,\;v_{\sigma}\in \R\Big\}. 
	\end{aligned}
\end{equation*}
The interpolator defined in \eqref{eq:def.ID} below exposes the meaning behind these values: the value $\bv_s$ represents a displacement at node $s$, while the face-bubble value $v_\sigma$ is a correction of the normal displacement to the face computed by averaging the nodal displacements.
To take into account homogeneous Dirichlet boundary conditions on $\partial\Omega$, we also consider the subspace
\begin{equation*}
  \begin{aligned}
  \UDz = \Big\{\bv_\D \in \UD: \bv_{s}=\mathbf{0}\quad\forall s\in\nodes^{\text{ext}}\,,\; v_{\sigma}=0\quad\forall \sigma\in\faces^{\text{ext}}\Big\}.
  \end{aligned}
\end{equation*}

\subsection{Reconstruction operators in $\UDz$}\label{subsec:operators}

We introduce here two reconstruction operators on each face and cell of the mesh; the first one reconstruct a gradient, the second is a displacement. These reconstructions are inspired by those in the lowest-order DDR method. At this order and disregarding the role of the scalar face unknowns, they also correspond to the standard projections in the Virtual Element Method \cite{Coulet.ea:20}; we also note that the addition of (vector) face DOFs to a (virtual element-like) fully discrete method has already been considered in \cite{droniou2023bubble,MR4886366}, but only along certain faces and to handle mechanical models in fractured media, not as a tool to recover a full {locking-free} feature.

Let $\sigma\in\faces$, and select nonnegative weights $(\omega_s^\sigma)_{s\in\nodes_\sigma}$ to express the center of mass $\overline{\mathbf{x}}_\sigma$ of $\sigma$ as a convex combination of its vertices:
\begin{equation*}
\overline{\mathbf{x}}_{\sigma} = \sum_{s \in \nodes_{\sigma} } \omega_{s}^{\sigma} \x_s\,,\quad
\sum_{s \in \nodes_{\sigma} } \omega_{s}^{\sigma}=1.
\end{equation*}
Then, the tangential face gradient $\nabla^\sigma:\UDz\to \Poly{0}(\sigma)^{d\times d}$ and tangential displacement reconstruction $\Pi^{\sigma}:\UDz\to \Polygras{1}(\sigma)^d$ are defined by: for all $\bv_\D\in\UDz$, 
\begin{equation*}
\begin{aligned}
\nabla^{\sigma} \bv_{\D} ={}& \frac{1}{|\sigma|} \sum_{e= s_1 s_2 \in \edges_{\sigma}} |e| {\mathbf{v}_{s_1}+\mathbf{v}_{s_2} \over 2} \otimes \normal_\sige,\\
\Pi^{\sigma} \bv_{\D}(\mathbf{x})  ={}& \nabla^{\sigma} \bv_{\D} (\mathbf{x} - \overline{\mathbf{x}}_{\sigma}) + \overline{\bv}_{\sigma}\quad\forall\mathbf{x}\in\sigma,\quad\mbox{ where } \overline{\bv}_{\sigma}=\sum_{s \in \nodes_{\sigma} } \omega_{s}^{\sigma}\bv_{s}.
\end{aligned}
\end{equation*}
In the definition above, $\edges_\sigma$ is the set of edges of $\sigma$ and $\normal_\sige$ is the unit normal vector to $e \in \edges_{\sigma}$ in the plane $\sigma$ oriented outward from $\sigma$. We write $e=s_1s_2$ to indicate that the vertices of $e$ are $s_1$ and $s_2$. The symbol $\otimes$ represents the tensor product of two vectors: $\mathbf{a}\otimes\mathbf{b}$ is the matrix such that $ (\mathbf{a}\otimes\mathbf{b})_{i,j}=\mathbf{a}_i\mathbf{b}_j$. We note in passing the following formula:
\begin{equation}\label{eq:average.Pisigma}
\frac{1}{|\sigma|}\int_\sigma \Pi^{\sigma} \bv_{\D}=\overline{\bv}_\sigma.
\end{equation}

Similarly, for each cell $K\in\cells$, we introduce nonnegative weights $(\omega_s^K)_{s\in\nodes_K}$ such that 
$$
\overline{\mathbf{x}}_K=\sum_{s\in\nodes_K}\omega_s^K\x_s\,,\quad 
\sum_{s \in \nodes_K } \omega_{s}^K=1,
$$
and we define the gradient reconstruction $\nabla^K:\UDz\to\Poly{0}(K)^{d\times d}$ and the displacement reconstruction $\Pi^K:\UDz\to\Poly{1}(K)^d$ by: for all $\bv_\D\in\UDz$,
\begin{equation}
\begin{aligned}
  \label{eq:def.nablaK}
  \nabla^{K} \bv_{\D}  ={}& \frac{1}{|K|} \sum_{\sigma \in \faces_{K}}  |\sigma| \overline{\bv}_{\sigma} \otimes \normal_\Ksig + \frac{1}{|K|} \sum_{\sigma \in \faces_{K}}  |\sigma| {v}_{\sigma} \normal_\sigma \otimes \normal_\Ksig, \\
  \Pi^{K} \bv_{\D}(\mathbf{x})  ={}& \nabla^{K} \bv_\D (\mathbf{x} - \overline{\mathbf{x}}_{K}) + \overline{\bv}_{K}\quad\forall\mathbf{x}\in K\,,\quad
  \mbox{ where }\overline{\bv}_{K} = \sum_{s \in \nodes_K } \omega_{s}^K \bv_{s}.
\end{aligned}
\end{equation}

Patching these local operators, we obtain their global (discontinuous but piecewise polynomial) counterparts $\nabla^\D:\UDz\to\Poly{0}(\cells)^{d\times d}$ and $\Pi^\D:\UDz\to\Poly{1}(\cells)^d$. We also define a piecewise constant displacement reconstruction operator $\wt{\Pi}^{\D}:\UDz\to\Poly{0}(\cells)^d$ by projection on piecewise constant functions. We therefore set: for all $\bv_\D\in\UDz$ and all $K\in\cells$,
\[
(\nabla^\D\bv_\D)_{|K}=\nabla^K\bv_\D\,,\quad
(\Pi^\D\bv_\D)_{|K}=\Pi^K\bv_\D\,,\quad
(\wt{\Pi}^\D\bv_\D)_{|K}=\overline{\bv}_K.
\]
Finally, the discrete symmetric gradient $\bbeps_\D$, divergence $\divScal_\D$, and stress tensor $\bbsigma_\D$ are defined by
\begin{alignat}{2}
\bbeps_{\D}&=\frac{1}{2}\left(\nabla^{\D}+ \prescript{t}{}{\nabla^{\D}}\right),\nonumber \\
\divScal_{\D} &= \text{Tr}\left( \bbeps_{\D}\right)=\frac{1}{|K|} \sum_{\sigma \in \faces_{K}}  |\sigma| \overline{\bv}_{\sigma} \cdot \normal_\Ksig + \frac{1}{|K|} \sum_{\sigma \in \faces_{K}}  |\sigma| {w}_\Ksig {v}_{\sigma}, \label{div_def_1} \\ 
\bbsigma_{\D}(\cdot)&=2\mu\bbeps_{\D}(\cdot) + \la \divScal_{\D}(\cdot) \mathbb{I}, \label{bbsig_def}
\end{alignat}
where $\text{Tr}$ is the matrix trace operator.
\subsection{Interpolator}\label{subsec:interpolator}
The space $\mathbf{\mathcal{C}}^0_0(\overline{\Omega})$ is spanned by the continuous functions $\overline{\Omega}\to\R^d$ that vanish on $\partial\Omega$. The interpolator $\IUD:\mathbf{\mathcal{C}}^0_0(\overline{\Omega})\to\UDz$ is defined by setting, for $\bv\in \mathbf{\mathcal{C}}^0_0(\overline{\Omega})$,
\begin{equation}\label{eq:def.ID}
  \begin{aligned}
  (\IUD \bv )_{s}={}&\bv(\x_s)&&\quad \forall s \in \nodes,\\
  (\IUD \bv)_{\sigma}={}&\frac{1}{|\sigma|} \int_{\sigma} \left( \gamma^{\sigma} \bv - \Pi^{\sigma}(\IUD\bv) \right)\cdot \normal_\sigma &&\quad\forall\sigma \in \faces.
  \end{aligned}
\end{equation}
Using \eqref{eq:average.Pisigma}, the second relation in \eqref{eq:def.ID} can be recast
\begin{equation*}
(\IUD \bv)_{\sigma}=\frac{1}{|\sigma|} \int_{\sigma} \gamma^{\sigma} \bv\cdot\normal_\sigma - \sum_{s\in\nodes_\sigma}\omega_s^\sigma \bv(\x_s)\cdot \normal_\sigma,
\end{equation*}
which shows that $(\IUD \bv)_{\sigma}$ is well-defined despite its apparent self-reference to $\IUD\bv$ in \eqref{eq:def.ID}.

\subsection{Variational formulation}\label{subsec:mixed}
We now introduce the numerical scheme for the variational formulation of 
\eqref{Lagrange_meca_contactfriction_1}: Find $ \bu_{\D} \in \UDz $ such that, for all $\bv_{\D} \in \UDz$,
\begin{equation}
\label{mixed_discrete}
\int_{\Omega} \bbsigma_{\D} (\bu_{\D}): \bbeps_{\D}(\bv_{\D}) + \mu_1 S_{\D}( \bu_{\D}, \bv_{\D})  = \int_\Omega  \mathbf{f} \cdot \wt \Pi^{\D} \bv_{\D}. 
\end{equation}
Here, the stabilisation bilinear form $S_{\D}$ is defined as
\begin{equation*}
S_{\D}( \bu_{\D}, \bv_{\D}) = \sum_{K\in\cells}S_K(\bu_\D,\bv_D)
\end{equation*}
with local stabilisation bilinear form $S_K:\UDz\times\UDz \to\R$ given by
\begin{equation}\label{eq:def.SK}
\begin{aligned}
S_K(\bu_\D,\bv_\D)={}& h_{K}^{d-2}  \sum_{s \in \nodes_K} \left(\bu_{\Ks} -\Pi^{K}\bu_{\D} (\x_s) \right) \cdot \left( \bv_{\Ks} -\Pi^{K}\bv_{\D}  (\x_s) \right)+ h_{K}^{d-2} \sum_{\sigma \in \faces_{K}}u_{\sigma}v_{\sigma}.
\end{aligned}
\end{equation}

\section{Main results}\label{sec:results}
In this section, we state the well-posedness of the scheme \eqref{mixed_discrete}, and the corresponding error estimates.
The error estimates will be stated in the following discrete $H^1$-seminorm.

\begin{definition}[Discrete $H^1$-like semi-norm on $\UD$]
The semi-norm $\NORM{1,\D}{{\cdot}}$ on $\UD$ is defined by: for all $\bu_\D\in \UD$,
\begin{equation}\label{eq:def.normD}
\NORM{1,\D}{\bu_\D}=\left(\sum_{K\in\cells}\NORM{1,K}{\bu_\D}^2\right)^{1/2}
\mbox{ with } \NORM{1,K}{\bu_\D}=\left(\NORM{L^2(K)}{\nabla^K\bu_\D}^2+S_K(\bu_\D,\bu_\D)\right)^{1/2},
\end{equation}
where $\nabla^K$ is defined by \eqref{eq:def.nablaK} and $S_K$ is given by \eqref{eq:def.SK}. Restricted to $\UDz$, $\NORM{1,\D}{{\cdot}}$ is actually a norm.
\end{definition}

To state the error estimates, we introduce the following notations:
\begin{itemize}
\item The (primal) consistency error is: for $\bv_\D\in\UDz$,
\begin{equation}\label{def:CD}
	C_{\D}(\bu,\bv_{\D}) = \left(\NORM{L^2(\Omega)}{\grad \bu - \grad^{\D} \bv_{\D}}^2 +  S_{\D}( \bv_{\D},  \bv_{\D}) \right)^{1/2}.
\end{equation}
\item Define $\Sigmaspace=\mathbf{H}_{\divScal}(\Omega; \mathcal{S}^d(\R))$,
 where $\mathcal{S}^d(\R)$ is the space of symmetric $d\times d$ real matrices. The adjoint consistency error (or limit-conformity measure) is then given, for $\bbsigma\in\Sigmaspace$, by 
\begin{equation}\label{def:WD}
\begin{aligned}
\mathcal{W}_{\D}(\bbsigma) ={}& \sup_{\bv_\D\in\UDz}\frac{w_{\D}(\bbsigma,\bv_{\D})}{\NORM{1,\D}{\bv_\D}},\\
\mbox{where } w_{\D}(\bbsigma,\bv_{\D}) ={}&\int_{\Omega} \bbsigma: \bbeps_{\D}(\bv_{\D} )+ \int_{\Omega}   \wt{\Pi}^{\D}\bv_{\D}\cdot \div \bbsigma.
\end{aligned}
\end{equation}
\end{itemize}

\begin{theorem}[Abstract error estimate]\label{abstr.err.est}
The numerical scheme \eqref{mixed_discrete} has a unique solution  $\bu_{\D} \in \UDz$ and, if $\bu$ is the solution of \eqref{Lagrange_meca_contactfriction_1}, we have the following abstract error estimate:	
\begin{equation}\label{estimat_error}
\begin{aligned}
\mu_1\NORM{L^2(\Omega)}{\grad^{\D}\bu_\D - \grad \bu} &\lsim   \mathcal{W}_{\D}(\bbsigma(\bu)) +\mu_{2}C_\D(\bu,\IUD\bu),
\end{aligned}
\end{equation}
where we recall that the hidden constant in $\lesssim$ is independent of the Lamé coefficients $\mu$ and $\la$.
\end{theorem}
\begin{proof}
See Section \ref{sec:proof.abstract.error}.
\end{proof}

In the following theorem, we denote by $\bs{H}^2(\cells)$ the space of vector-valued functions defined on $\Omega$ that are $\bs{H}^2$ on each $K\in\cells$. This space is endowed with its usual broken semi-norm. 

\begin{theorem}[Error estimate]\label{error_estimate}
 Let $\bu$ be the solution to \eqref{Lagrange_meca_contactfriction_1} and assume that $\bu \in \bs{H}^2(\cells)$. Then, the solution $\bu_{\D}$ of \eqref{mixed_discrete} satisfies the following error estimate:
$$
\mu_1\NORM{L^2(\Omega)}{\grad^{\D}\bu_\D - \grad \bu} \lsim h(\mu_2\SEMINORM{H^2(\cells)}{\bu}+\la\SEMINORM{H^1(\cells)}{\divScal\bu}), 
$$
where the coefficient hidden in $\lsim$ is independent of Lam\'e coefficients $\mu$ and $\la$. 
\end{theorem}
\begin{proof}
See Section \ref{sec:proof.error}.
\end{proof}

\begin{remark}[Locking-free estimate]\label{robust_rem}
Recall the following regularity estimate for the displacement $\bu$ solution to \eqref{eq:model.strong} (see \cite[Lemma~$2.3$]{brenner1992linear}):
\begin{equation}
  \NORM{\mathbf{H}^2(\Omega)}{\bu}+\la\SEMINORM{H^1(\Omega)}{\divScal\, \bu} \lesssim \NORM{L^{2}(\Omega)}{\mathbf{f}},\label{reg1}
\end{equation}
where the hidden constant depends only on $\Omega$. The regularity assumption \eqref{reg1} holds for a sufficiently smooth domain $\Omega$ or if $\Omega$ is a convex polygon in two dimensions \cite{brenner1992linear, vogelius1983analysis}. In three-dimensional settings, however, it is very technical to derive such regularity; see \cite{grisvard1992singularities} for further details.

Combining the regularity property \eqref{reg1} and Theorem~\ref{error_estimate}, we obtain the following estimate, in which the hidden constant and the right-hand side are independent of $\la$:
$$\mu_{1}\NORM{L^2(\Omega)}{\grad^{\D}\bu_\D - \grad \bu} \lsim h\left(1+\mu_2\right)\NORM{L^2(\Omega)}{\mathbf{f}}.$$
\end{remark}
\begin{remark}[Comparison with existing methods]
A summary of the local degrees of freedom required by different mixed finite element formulations (on simplices) and virtual element methods is presented in Tables~\ref{tab:dof_comparison}-\ref{tab:dof_comparison_VEM}, which highlights the compactness of our low-order scheme compared with existing methods.
\begin{table}[h]
\centering
\scriptsize
\caption{Comparison of the number of local DOFs on simplices for various mixed finite element methods.}
\begin{tabular}{lcc}
\toprule
\textbf{Method} & \textbf{Polynomial Degree}  & \textbf{Total DoFs per Element} \\ 
\midrule
Falk~\cite{falk2008finite}~(2D) & Quadratic & 15 (12 stress + 3 displacement) \\
Adams \& Cockburn~\cite{adams2004mixed}~(3D) & Quartic & 162 (stress space) \\
Huang {et al.}~\cite{huang2024new}~(2D and 3D) & Lowest-order  &18~(2D)/48~(3D) \\
\textbf{Present work}~(2D and 3D) & Lowest-order & 9~(2D)/16~(3D) \\
\bottomrule
\end{tabular}
\label{tab:dof_comparison}
\end{table}

\begin{table}[h]
\centering
\scriptsize
\caption{Comparison of the number of local DOFs  for various virtual finite element methods.}
\begin{tabular}{lccc}
\toprule
\textbf{Method} &\textbf{Element}& \textbf{Polynomial Degree}   & \textbf{Total DoFs per Element} \\ 
\midrule
Huang {et al.}~\cite{Huang-2023}~(2D) & Pentagon &Lowest-order & 20 \\
Lovadina \& Molinari~\cite{lovadina2026volumetric}~(2D) & Pentagon & Lowest-order & 16 \\
Huang {et al.}~\cite{Huang-2025}~(3D) & Hexahedron&Lowest-order  & 42 \\
\textbf{Present work}~(2D and 3D)& Pentagon~(2D)/Hexahedron~(3D) & Lowest-order & 15~(2D)/30~(3D) \\
\bottomrule
\end{tabular}
\label{tab:dof_comparison_VEM}
\end{table}
\end{remark}
\begin{remark}
In contrast to the mixed formulation proposed in \cite{lovadina2026volumetric}, which gives rise to semi-definite saddle-point systems, the present DDR-based discretization retains the symmetric structure of the present model problem. Although the number of degrees of freedom is comparable in two dimensions, the proposed approach offers improved numerical stability and a simpler implementation.
\end{remark}

\section{Proof of the error estimate}\label{sec:proofs}

This section provides the theoretical foundation for the main results by establishing key preliminary results.

\begin{lemma}[DOF-based bound on the discrete norm]
Let $K\in\cells$. Recalling the definition \eqref{eq:def.normD} of $\NORM{1,K}{{\cdot}}$, we have, for all $\bw_\D=((\bw_{s})_{s\in\nodes},(w_{\sigma})_{\,\sigma\in \faces})\in\UD$,
\begin{equation*}
\NORM{1,K}{\bw_\D}\lsim h_K^{-1}|K|^{1/2}\left(\max_{s\in\nodes_K}|\bw_{s}|+\max_{\sigma\in\faces_{K}}|w_{\sigma}|\right).
\end{equation*}
\end{lemma}

\begin{proof} 
The proof follows from \cite[Lemma 5.4]{MR4886366} applied to $\bw_K\coloneq((\bw_{s})_{s\in\nodes_K},(w_{\sigma}\normal_\sigma)_{\,\sigma\in \faces_K})$ and with $\mathcal F^+_{\Gamma,K}=\faces_K$.
\end{proof}
\begin{theorem}[Discrete Korn inequality]\label{th:korn}
It holds
\begin{equation}\label{eq:Korn}
\NORM{1,\D}{\bv_\D}^2\lsim \NORM{L^2(\Omega)}{\bbeps_\D(\bv_\D)}^2 + S_\D(\bv_\D,\bv_\D)\qquad\forall \bv_\D\in\UDz.
\end{equation}
\end{theorem}
\begin{proof}
The proof follows the arguments of \cite[Theorem 5.7]{MR4886366}, except that, since no fractures are present here, the node-averaging operator from \cite[Section 7.3.2]{hho-book} can be used without modification. 
\end{proof}

\subsection{Commutation property}
The following lemma states that the interpolator $\IUD$ is a Fortin operator with respect to the divergence. It is classically the key to obtain a {locking-free} method, and hinges on the additional face degrees of freedom.
\begin{lemma}[Commutation property]\label{lm:com}
It holds, for all $\bu\in\mathbf{\mathcal{C}}^0_0(\overline{\Omega})$,
\begin{equation*}
\int_{K}\divScal_{\D}\left(\IUD \bu\right)\mathbb{I}:\zeta= \int_{K}\left(\divScal\bu\right) \mathbb{I}:\zeta \qquad \forall \zeta\in \Poly{0}(K)^{d\times d}. 
\end{equation*}
\end{lemma}
\begin{proof}
Set $\bv_{\D}=\IUD \bu$. The definition of $\divScal_{\D}$  in \eqref{div_def_1} gives
\begin{align}\label{div_trace_1}
(\divScal_\D\bv_\D)|_K = 
\frac{1}{|K|} \sum_{\sigma \in \faces_{K}}  |\sigma| \overline{\bv}_{\sigma} \cdot\normal_\Ksig+\frac{1}{|K|} \sum_{\sigma \in \faces_{K}}|\sigma| w_{\Ksig} v_{\sigma}. 
\end{align}
By \eqref{eq:average.Pisigma}  and \eqref{eq:def.ID} we have
\[
\overline{\bv}_{\sigma}= \frac{1}{|\sigma|}\int\limits_{\sigma} \Pi^{\sigma}\bv_\D\quad\text{ and }\quad
v_{\sigma}=  \frac{1}{|\sigma|}\int\limits_{\sigma}\left(\gamma^{\sigma}\bu- \Pi^{\sigma}\bv_\D\right)\cdot\normal_\sigma.
\]
Plugged into \eqref{div_trace_1} this gives
\begin{align*}
\divScal_{\D}(\IUD \bu )&=\frac{1}{|K|} \sum_{\sigma \in \faces_{K}}  \left(\int\limits_{\sigma} \Pi^{\sigma}\bv_\D \right) \cdot\normal_\Ksig +\frac{1}{|K|} \sum_{\sigma \in \faces_{K}} w_{\Ksig} \left(\int\limits_{\sigma} \gamma^\sigma\bu-\int\limits_{\sigma} \Pi^{\sigma}\bv_\D\right)\cdot\normal_\sigma\\
\overset{w_{K\sigma}\normal_{K\sigma}=\normal_{\sigma}}&=\frac{1}{|K|} \sum_{\sigma \in \faces_{K}}\int\limits_{\sigma} \gamma^\sigma\bu\cdot\normal_{K\sigma}\\
&=\int\limits_{K} \divScal\, \bu,
\end{align*}
where we have used the divergence theorem in the conclusion. The conclusion of the lemma follows by multiplying this equality by the (constant) scalar number $\mathbb{I}:\zeta$.
\end{proof}
\begin{remark}
A scalar face bubble is added into the lowest-order DDR discrete space, which ensures the commutation property stated in Lemma~\ref{lm:com}. This property is key to obtaining a locking-free scheme in the present setting.
\end{remark}
\subsection{Proof of the abstract error estimate (Theorem  \ref{abstr.err.est})}\label{sec:proof.abstract.error}
We first introduce the discrete energy inner product $\IPbracket{\cdot}{\cdot}_{e,\D}$, defined for $\bu_{\D},\bv_{\D} \in \UDz$ as follows
\begin{equation}\label{discrete:energy:ps}
\IPbracket{\bu_{\D}}{\bv_{\D}}_{e,\D} = \int_{\Omega} \bbsigma_{\D} (\bu_{\D}): \bbeps_{\D}(\bv_{\D} )  + \mu_1 S_{\D}( \bu_{\D}, \bv_{\D} ) 
\end{equation}	
 and denote by $\NORM{e,\D}{{\cdot}}$ its associated norm. By definitions \eqref{bbsig_def} of $\bbsigma_{\D}$ and \eqref{div_def_1} of $\divScal_\D$, we have, for all $\bv_\D\in\UD$,
 \[
 \bbsigma_{\D} (\bv_{\D}): \bbeps_{\D}(\bv_{\D})\ge 2\mu |\bbeps_\D(\bv_\D)|^2+\lambda(\divScal_{\D}\bv_\D)^2\ge 2\mu_1 |\bbeps_\D(\bv_\D)|^2.
\]
 The discrete Korn inequality \eqref{eq:Korn} then yields
 \begin{equation}\label{relat.norme.discrete}
\mu_1 \NORM{1,\D}{\bv_{\D}}^2 \lsim \NORM{e,\D}{\bv_{\D}}^2\quad \forall \bv_{\D} \in \UDz.
\end{equation} 
Since $\bu$ is a weak solution of the linear elasticity problem, we have $-\div(\bbsigma(\bu))=\mathbf{f}\in \mathbf{L}^2(\Omega)$. The definition \eqref{def:WD} of $w_\D$ gives, for all $\bw_\D\in\UDz$
\begin{equation}\label{cons.dual_equ_1}
\int_{\Omega} \bbsigma(\bu): \bbeps_{\D}(\bw_{\D} ) -  \int_{\Omega}  \mathbf{f} \cdot \wt{\Pi}^{\D}\bw_{\D} =  w_{\D}(\bbsigma(\bu),\bw_{\D}).
\end{equation}	 
Subtracting \eqref{mixed_discrete} (with $\bv_\D=\bw_\D$) from \eqref{cons.dual_equ_1}, we obtain
\begin{equation}\label{cons.dual.somme.var.formul}
	\int_{\Omega} (\bbsigma(\bu)-\bbsigma_{\D}(\bu_{\D})): \bbeps_{\D}(\bw_{\D} ) - \mu_1S_{\D}( \bu_{\D}, \bw_{\D} )   =   w_{\D}(\bbsigma(\bu),\bw_{\D}).
\end{equation}	 
Take $\bv_{\D} \in \UDz$ and set $\bw_{\D} = \bv_{\D} - \bu_{\D}$ in \eqref{cons.dual.somme.var.formul} to get
\begin{align}\label{estima.1}
\NORM{e,\D}{\bv_{\D} - \bu_{\D}}^2 &= w_{\D}(\bbsigma(\bu),\bv_{\D}-\bu_{\D}) -\int_{\Omega} (\bbsigma(\bu)-\bbsigma_{\D}(\bv_{\D})): \bbeps_{\D}(\bv_{\D}-\bu_{\D}) \nonumber\\
&  \quad+ \mu_1 S_{\D}(\bv_{\D},\bv_{\D}-\bu_{\D} )\nonumber\\
\overset{\eqref{eq:model.strong}, \eqref{bbsig_def}}&= w_{\D}(\bbsigma(\bu),\bv_{\D}-\bu_{\D}) -2\mu \int_{\Omega} (\bbeps(\bu)-\bbeps_{\D}(\bv_{\D})):\bbeps_{\D}(\bv_{\D}-\bu_{\D}) \nonumber\\
 &\quad-\la \int_{\Omega} ((\divScal\, \bu )-\divScal_{\D}(\bv_{\D} )) \mathbb{I}: \bbeps_{\D}(\bv_{\D}-\bu_{\D}) 
  + \mu_1 S_{\D}(\bv_{\D},\bv_{\D}-\bu_{\D}).
\end{align}
Note that $\bu\in \mathbf{\mathcal{C}}^0_0(\overline{\Omega})$, see \cite[Theorem~$7.97$ on page~$493$]{salsa2016partial}, which justifies that we can choose $\bv_{\D}:=\IUD \bu$ in \eqref{estima.1}. Using Lemma \ref{lm:com}~(on each $K\in\cells$, with $\zeta=\bbeps_{\D}(\bv_{\D}-\bu_{\D})|_K$) leads to
\begin{align*}
\NORM{e,\D}{\bv_{\D}-\bu_{\D}}^2 ={}& w_{\D}(\bbsigma(\bu),\bv_{\D}-\bu_{\D}) -2\mu \int_{\Omega} (\bbeps(\bu)-\bbeps_{\D}(\bv_{\D})):\bbeps_{\D}(\bv_{\D}-\bu_{\D}) \nonumber\\& + \mu_1 S_{\D}(\bv_{\D},\bv_{\D}-\bu_{\D}). 
\end{align*}
Invoking the norm estimate \eqref{relat.norme.discrete}, the definitions \eqref{def:WD} of $\mathcal W_\D$ and \eqref{def:CD} of $C_\D$ as well as Cauchy--Schwarz inequalities, we obtain
\begin{align*}
\mu_1 \NORM{1,\D}{\bv_{\D} - \bu_{\D}}^2\lesssim{}& \mathcal{W}_\D(\bbsigma(\bu))\NORM{1,\D}{\bv_{\D} - \bu_{\D}}
+\mu_2C_\D(\bu,\bv_\D)\NORM{1,\D}{\bv_{\D} - \bu_{\D}}\\
&+\mu_1 S_\D(\bv_\D,\bv_\D)^{1/2}S_\D(\bv_\D-\bu_\D,\bv_\D-\bu_\D)^{1/2}\\
\lesssim{}&\mathcal{W}_\D(\bbsigma(\bu))\NORM{1,\D}{\bv_{\D} - \bu_{\D}}
+\mu_2C_\D(\bu,\bv_\D)\NORM{1,\D}{\bv_{\D} - \bu_{\D}},
\end{align*}
where the conclusion follows from $\mu_1\le \mu_2$, $S_\D(\bv_\D,\bv_\D)^{1/2}\le C_\D(\bu,\bu_\D)$ and $S_\D(\bv_\D-\bu_\D,\bv_\D-\bu_\D)^{1/2}\le
\NORM{1,\D}{\bv_\D-\bu_\D}$.
Simplifying by $\NORM{1,\D}{\bv_{\D} - \bu_{\D}}$, we infer
\[
\mu_1 \NORM{1,\D}{\bv_{\D} - \bu_{\D}}\lesssim\mathcal{W}_\D(\bbsigma(\bu))
+\mu_2C_\D(\bu,\bv_\D).
\]
The estimate \eqref{estimat_error} then follows by using a triangle inequality to get
\begin{align*}
\NORM{L^2(\Omega)}{\nabla^\D\bu_\D-\nabla \bu}\le{}&
\NORM{L^2(\Omega)}{\nabla^\D\bu_\D-\nabla^D \bv_\D}+\NORM{L^2(\Omega)}{\nabla^\D\bv_\D-\nabla \bu}\\
\le{}& \NORM{1,\D}{\bu_\D-\bv_\D}+C_\D(\bu,\bv_\D)
\end{align*}
and by recalling that $\bv_\D=\IUD\bu$. 

\subsection{Proof of the error estimate (Theorem \ref{error_estimate})}\label{sec:proof.error}

Theorem \ref{error_estimate} directly follows from the abstract error estimate \eqref{estimat_error}, and Lemmas \ref{lem:consistency.nabla} and \ref{lem:adjoint.consistency} below.
\begin{lemma}[Consistency of the gradient reconstruction]\label{lem:consistency.nabla}
If $\bu\in \U_0\cap \mathbf{H}^2(\cells)$ then, recalling the definition \eqref{def:CD} of $C_\D$, it holds
\begin{equation*}
C_\D(\bu,\IUD\bu)\lesssim h\SEMINORM{H^2(\cells)}{{\bu}}.
\end{equation*}
\end{lemma}
\begin{proof}
The proof is analogous to that of \cite[Theorem 5.8]{MR4886366}, except that our domain contains no fractures.
\end{proof}

\begin{lemma}[Adjoint consistency] \label{lem:adjoint.consistency}
If $\bu \in H^2(\cells)$ then, recalling the definition \eqref{def:WD} of the adjoint consistency error $\mathcal{W}_\D$, it holds
\begin{equation*}
\mathcal{W}_{\D}(\bbsigma(\bu)) \lsim  h\left(\mu_2\SEMINORM{H^2(\cells)}{\bu}+\la\SEMINORM{H^1(\cells)}{\divScal\bu}\right),
\end{equation*} 
where the coefficient hidden in $\lsim$ is independent of Lam\'e coefficients $\mu$ and $\la$.
\end{lemma}
\begin{proof}
The definition \eqref{def:WD} of $w_\D$ yields 
	$$
\begin{aligned}
w_{\D}(\bbsigma(\bu),\bv_{\D}) = &\sum_{K \in \cells}\left( \Big[\int_{K} \bbsigma(\bu) \Big] : \bbeps_{K}(\bv_{\D} ) +\sum_{\sigma \in \faces_{K}} \overline{\bv}_K \cdot  \int_{\sigma} (\bbsigma(\bu)|_K\, \normal_{\Ksig}) \right)\\
= &\sum_{K \in \cells}  \left(|K|\bbtau_K : \bbeps_{K}(\bv_{\D} ) +\sum_{\sigma \in \faces_{K}} |\sigma|\,\overline{\bv}_K \cdot  \mathbf{\tau}_{\Ksig} \right)
\end{aligned}
$$ with $\bbeps_K(\bv_\D)=(\bbeps_\D(\bv_\D))|_K$ and
$$
\bbtau_K = \frac{1}{|K|} \int_K \bbsigma(\bu) \quad \text{and} \quad \mathbf{\tau}_{K\sigma} = \frac{1}{|\sigma|} \int_{\sigma} \left(\bbsigma(\bu)|_K \normal_{K \sigma}\right).
  $$
	Noticing that $\bbtau_K:\bbeps_K(\bv_{\D}) = \bbtau_K:\grad^K \bv_{\D}$ (since $\bbtau_K$ is symmetric) and recalling the definition \eqref{eq:def.nablaK} of  $\grad^K$, we infer
\begin{equation}\label{eq:adjoint.consistency.1}
\begin{aligned}
w_{\D}(\bbsigma(\bu),\bv_{\D}) =& \sum_{K \in \cells} \sum_{\sigma \in \faces_K} \left(|\sigma|\, \overline{\bv}_{\sigma} \cdot ( \bbtau_K \normal_{\Ksig})+|\sigma|\, {v}_{\sigma}\normal_{\sigma}\cdot( \bbtau_K \normal_{\Ksig}) + |\sigma|\, \overline{\bv}_{K} \cdot \mathbf{\tau}_{K\sigma} \right).
\end{aligned}
\end{equation}
Moreover, as $\sum_{\sigma \in \faces_K} |\sigma|\,\normal_{\Ksig}=0$ for all $K\in\cells$,
	\begin{equation}\label{eq:zero.1}
		\sum_{K \in \cells} \sum_{\sigma \in \faces_K} \, |\sigma|\overline{\bv}_{K} \cdot \left( \bbtau_{K} \normal_{\Ksig}\right) = \sum_{K \in \cells} \overline{\bv}_{K} \cdot \left(\bbtau_{K} \sum_{\sigma \in \faces_K} |\sigma|\,\normal_{\Ksig}\right)=0.
	\end{equation}
	We also note that, by regularity of $\bu$, $\mathbf{\tau}_{K\sigma}+\mathbf{\tau}_{L\sigma}=0$ whenever $\sigma$ is a face between two cells $K, L$. Since $\overline{\bv}_\sigma=0$ and $v_\sigma=0$ if $\sigma$ is a boundary face (since $\bv_\D\in\UDz$), gathering the sums by face we infer that
	\begin{equation}\label{eq:zero.2}
    \sum_{K \in \cells} \sum_{\sigma \in \faces_K} |\sigma|\, \overline{\bv}_{\sigma} \cdot \mathbf{\tau}_{K\sigma}=0
    \quad\text{ and }\quad
    \sum_{K \in \cells} \sum_{\sigma \in \faces_K} |\sigma|\, v_{\sigma} \normal_\sigma\cdot \mathbf{\tau}_{K\sigma}=0.
	\end{equation}
Subtracting \eqref{eq:zero.1} and \eqref{eq:zero.2} from \eqref{eq:adjoint.consistency.1} leads to
\begin{align*}
w_{\D}&(\bbsigma(\bu),\bv_{\D})\\
 ={}&
  \sum_{K \in \cells} \sum_{\sigma \in \faces_K} |\sigma|\, \overline{\bv}_{\sigma} \cdot \left(\bbtau_{K}\normal_{\Ksig} - \mathbf{\tau}_{K\sigma}\right) 
    +\sum_{K \in \cells} \sum_{\sigma \in \faces_K}|\sigma|\, {v}_{\sigma}\normal_{\sigma}\cdot\left(\bbtau_K \normal_{\Ksig} - \mathbf{\tau}_{K\sigma}\right)\\
  &+ \sum_{K \in \cells} \sum_{\sigma \in \faces_K} |\sigma|\, \overline{\bv}_{K} \cdot \left( \mathbf{\tau}_{K\sigma} - \bbtau_{K}\normal_{\Ksig}\right) \\ 
={}&\sum_{K \in \cells} \sum_{\sigma \in \faces_K} |\sigma|\, \left(\overline{\bv}_{\sigma} - \overline{\bv}_{K} \right) \cdot \left(\mathbf{\tau}_{K\sigma} - \bbtau_{K}\normal_{\Ksig}\right)
+ \sum_{K \in \cells} \sum_{\sigma \in \faces_K} |\sigma|\, {v}_{\sigma}\normal_{\sigma}\cdot \left(\bbtau_{K}\normal_{\Ksig}-\mathbf{\tau}_{K\sigma} \right).
\end{align*}
Using the Cauchy--Schwarz inequality and invoking Lemma \ref{inequ_pour_consist} below, we infer
  \begin{align}
w_{\D}(\bbsigma(\bu),\bv_{\D}) \le{}& \left(\sum_{K \in \cells} \sum_{\sigma \in \faces_K} \frac{|\sigma|}{h_K} |\overline{\bv}_{\sigma}-\overline{\bv}_{K} |^2\right)^{\nicefrac12} \left(\sum_{K \in \cells} \sum_{\sigma \in \faces_K} |\sigma|h_K|\mathbf{\tau}_{K\sigma} - \bbtau_{K}\normal_{\Ksig} |^2\right)^{\nicefrac12}\nonumber\\
& + \left(\sum_{K \in \cells} \sum_{\sigma \in \faces_K} \frac{|\sigma|}{h_K} |{v}_{\sigma}\normal_{\sigma} |^2\right)^{\nicefrac12} \left(\sum_{K \in \cells} \sum_{\sigma \in \faces_K} |\sigma|h_K|\mathbf{\tau}_{K\sigma} - \bbtau_{K}\normal_{\Ksig} |^2\right)^{\nicefrac12}\nonumber\\
\lesssim{}& \NORM{1,\D}{\bv_{\D}}  \left(\sum_{K \in \cells} \sum_{\sigma \in \faces_K} |\sigma| h_K |\mathbf{\tau}_{K\sigma} - \bbtau_{K}\normal_{\Ksig} |^2\right)^{\nicefrac12}.
  \label{est:limconf.1}
  \end{align}
  By \cite[Lemma B.6]{gdm}, we have
  \begin{align*}
\sum_{K \in \cells} \sum_{\sigma \in \faces_K}|\sigma| h_K |\mathbf{\tau}_{K\sigma} - \bbtau_{K}\normal_{\Ksig} |^2
&\lesssim\sum_{K \in \cells} h_K^2\SEMINORM{H^1(K)}{\bbsigma(\bu)}^2 \\ 
&\lesssim h^2 \left(\mu_2\SEMINORM{H^2(\cells)}{\bu}+\la\SEMINORM{H^1(\cells)}{\divScal\bu}\right)^2,
\end{align*}
where the conclusion follows from the definition of $\bbsigma$ in \eqref{eq:model.strong}.
Plugging this into \eqref{est:limconf.1}, dividing by $\NORM{1,\D}{\bv_{\D}}$,  and taking the supremum over $\bv_\D$ concludes the proof.
\end{proof}
\begin{lemma}\label{inequ_pour_consist}
For all $\bv_{\D} \in \UDz$, the following two inequalities hold:
\begin{align*}
\left(\sum_{K \in \cells} \sum_{\sigma \in \faces_K} \frac{|\sigma|}{h_K}|\overline{\bv}_{\sigma} - \overline{\bv}_{K}|^2 \right)^{\nicefrac12}
+\left(\sum_{K \in \cells} \sum_{\sigma \in \faces_K} \frac{|\sigma|}{h_K} |{v}_{\sigma}\normal_{\sigma} |^2\right)^{\nicefrac12}
\lesssim{}& \NORM{1,\D}{\bv_{\D}}.
\end{align*}
\end{lemma}
\begin{proof}	
The proof of this lemma follows from the same arguments as in \cite[Lemma 5.11]{MR4886366}. 
\end{proof}

\section{Frictionless contact mechanics model}\label{Tresca_friction_extension}

In this section, we show that our approach can be combined with the one in \cite{MR4886366} to design a {locking-free} scheme for contact mechanics. The model in \cite{MR4886366} is that of a fractured medium, in which a linear elastic law is considered in the matrix while a frictionless law is imposed at the fractures (see Remark \ref{rem:tresca} for an extension to contact problems with Tresca friction). We denote by $\Gamma$ the fracture network, made of a union of flat surfaces (the fractures), and by $\pm$ the two sides arbitrarily chosen for each fracture. The outer normal to the $+$ side of each fracture is denoted by $\normal^+$.
In strong form, the frictionless contact mechanics model is expressed as
	\begin{equation}\label{sec6:mode_tresca_1} 
		\left\{\!\!\!\!
		\begin{array}{lll}
			& -\div \bbsigma(\bu)= \mathbf{f}  & \mbox{ on } \Omega{\backslash}\ov\Gamma,\\[1ex]
			& \bbsigma(\bu)=2\mu\bbeps(\bu)+\la (\divScal\, \bu )\mathbb{I} & \mbox{ on }  \Omega{\backslash}\ov\Gamma,\\[1ex]
			& \gamma_\normal^+\bbsigma(\bu)+\gamma_\normal^-\bbsigma(\bu) = \mathbf{0} & \mbox{ on }  \Gamma,\\[1ex]
			& T_{\normal}(\bu) \leqslant 0,~ \jump{\bu }_{\normal} \leqslant 0,~ \jump{\bu }_{\normal}T_{\normal}(\bu) =0 & \mbox{ on }  \Gamma,\\[1ex]
			&\bu=\mathbf{0}  & \mbox{ on } \partial\Omega,
		\end{array}
		\right.
	\end{equation}
 where $\jump{\bu}$ is the jump of $\bu$ along the fractures (difference between the traces of $\bu$ on the positive and negative sides), $\jump{\bu}_\normal=\jump{\bu}\cdot\normal^+$ is its normal component, $\gamma_\normal^\pm\bbsigma(\bu)=\bbsigma(\bu)\normal^{\pm}$ are the traces on each side of the fracture, and the normal surface traction is defined as $T_{\normal}(\bu) = \gamma_\normal^+\bbsigma(\bu)\cdot \normal^+$.

To describe the numerical scheme for \eqref{sec6:mode_tresca_1}, we require some modification in the discrete space $\UD$ as in \cite{MR4886366}. Specifically, vertices and faces on the fracture networks can have multiple DOFs attached to them. Denote by $\cells_s$ the set of cells containing the vertex $s$ and by $\faces_\Gamma$ the trace of the mesh over $\Gamma$. The vertex DOFs are indexed by cell-vertex pairs $\mathcal{K}s$ (with $K\in\cells_s$), with $\bv_{\mathcal{K}s}$ denoting the nodal unknown at $s$ on the side of $K$ from the fracture; there is only one such nodal unknown for each side of the fracture around $s$, but it can differ form the nodal unknowns on the other side(s). Similarly, we put one scalar bubble DOF on each side of a fracture face: if $\sigma\in\faces_\Gamma\cap\faces_K$ then $v_{K\sigma}$ is the scalar unknown associated with $\sigma$ on the side of $K$ (there is another, different, unknown $v_{L\sigma}$ if $L$ is the cell on the other side of $K$ from $\sigma$). This is illustrated in Figure \ref{sec6:Dofs_fig}.

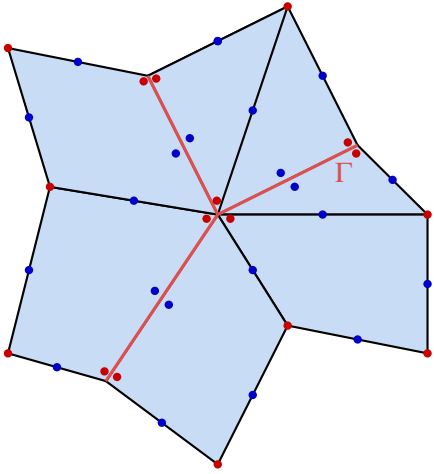
\begin{figure}[H]
  \centering
  \begin{minipage}[t][0pt][t]{0.48\textwidth}
    \vspace{0pt} 
    \begin{tikzpicture}[scale=1.84, baseline={(0,0)}] 
      \definecolor{cellcolor}{RGB}{200, 220, 245}
      \definecolor{fractureRed}{RGB}{220, 80, 80}
      \definecolor{nodeRed}{RGB}{200, 0, 0}
      \definecolor{bubbleBlue}{RGB}{0, 0, 200}
      
      \coordinate (center) at (2, 2);
      
      \draw[thick, fill=cellcolor] (2, 2) -- (3, 2.5) -- (2.5, 3.5) -- (1.5, 3) -- cycle; 
      \draw[thick, fill=cellcolor] (3.5, 2) -- (3, 2.5) -- (2, 2) -- cycle;
      \draw[thick, fill=cellcolor] (2, 2) -- (3.5, 2) -- (3.5, 1) -- (2.5, 1.2) -- cycle; 
      \draw[thick, fill=cellcolor] (2, 2) -- (2.5, 1.2) -- (2, 0.2) -- (1.2, 0.8) -- cycle; 
      \draw[thick, fill=cellcolor] (2, 2) -- (1.2, 0.8) -- (0.5, 1) -- (0.8, 2.2) -- cycle; 
      \draw[thick, fill=cellcolor] (2, 2) -- (0.8, 2.2) -- (0.5, 3.2) -- (1.5, 3) -- cycle; 
      \draw[thick, fill=cellcolor] (2, 2) -- (1.5, 3) -- (2.5, 3.5) -- cycle; 
      
      \draw[very thick, color=fractureRed] (2, 2) -- (3, 2.5);        
      \node[font=\small, color=fractureRed, font=\bfseries] at (2.9, 2.3) {$\Gamma$};
      
      \draw[very thick, color=fractureRed] (2, 2) -- (1.2, 0.8);      
      \draw[very thick, color=fractureRed] (2, 2) -- (1.5, 3);        
      
      \fill[nodeRed] (3-0.07, 2.52) circle (0.03);
      \fill[nodeRed] (3-0.01, 2.5-0.06) circle (0.03);
      \fill[nodeRed] (2.5, 3.5) circle (0.03);
      \fill[nodeRed] (1.5-0.03, 3-0.04) circle (0.03);
      \fill[nodeRed] (1.5+0.06, 3-0.02) circle (0.03);
      \fill[nodeRed] (3.5, 2) circle (0.03);
      \fill[nodeRed] (3.5, 1) circle (0.03);
      \fill[nodeRed] (2.5, 1.2) circle (0.03);
      \fill[nodeRed] (2, 0.2) circle (0.03);
      \fill[nodeRed] (1.2+0.08, 0.8+0.03) circle (0.03);
      \fill[nodeRed] (1.2-0.01, 0.8+0.07) circle (0.03);
      \fill[nodeRed] (0.5, 1) circle (0.03);
      \fill[nodeRed] (0.8, 2.2) circle (0.03);
      \fill[nodeRed] (0.5, 3.2) circle (0.03);
      
      
      


\path let \p1 = (2,2), \p2 = (3,2.5) in
  coordinate (mid1) at ($(\p1)!0.5!(\p2)$);
\coordinate (frac1a) at ($(mid1) + (-0.05,0.05)$);
\coordinate (frac1b) at ($(mid1) + (0.05,-0.05)$);
\fill[bubbleBlue] (frac1a) circle (0.03);
\fill[bubbleBlue] (frac1b) circle (0.03);

\path let \p1 = (2,2), \p2 = (1.2,0.8) in
  coordinate (mid2) at ($(\p1)!0.5!(\p2)$);
\coordinate (frac2a) at ($(mid2) + (-0.05,0.05)$);
\coordinate (frac2b) at ($(mid2) + (0.05,-0.05)$);
\fill[bubbleBlue] (frac2a) circle (0.03);
\fill[bubbleBlue] (frac2b) circle (0.03);

\path let \p1 = (2,2), \p2 = (1.5,3) in
  coordinate (mid3) at ($(\p1)!0.5!(\p2)$);
\coordinate (frac3a) at ($(mid3) + (0.05,0.05)$);
\coordinate (frac3b) at ($(mid3) + (-0.05,-0.06)$);
\fill[bubbleBlue] (frac3a) circle (0.03);
\fill[bubbleBlue] (frac3b) circle (0.03);

      \coordinate (fK1_1) at ($(2,2)!0.5!(2.5, 1.2)$); \fill[bubbleBlue] (fK1_1) circle (0.03);
      
      \coordinate (fK1_1) at ($(1.5,3)!0.5!(0.5, 3.2)$); \fill[bubbleBlue] (fK1_1) circle (0.03);
      
      \coordinate (fK1_1) at ($(0.5, 1)!0.5!(0.8, 2.2)$); \fill[bubbleBlue] (fK1_1) circle (0.03);
      \coordinate (fK1_1) at ($(2, 2)!0.5!(0.8, 2.2)$); \fill[bubbleBlue] (fK1_1) circle (0.03);
      \coordinate (fK1_1) at ($(2, 0.2)!0.5!(1.2, 0.8)$); \fill[bubbleBlue] (fK1_1) circle (0.03);
      \coordinate (fK1_1) at ($(2.5, 1.2)!0.5!(3.5, 1)$); \fill[bubbleBlue] (fK1_1) circle (0.03);
      \coordinate (fK1_1) at ($(2, 2)!0.5!(3.5, 2)$); \fill[bubbleBlue] (fK1_1) circle (0.03);
      \coordinate (fK1_1) at ($(2, 2)!0.5!(2.5, 3.5)$); \fill[bubbleBlue] (fK1_1) circle (0.03);
      \coordinate (fK1_1) at ($(2.5, 3.5)!0.5!(3, 2.5)$); \fill[bubbleBlue] (fK1_1) circle (0.03);
      \coordinate (fK1_1) at ($(2.5,3.5)!0.5!(1.5,3)$); \fill[bubbleBlue] (fK1_1) circle (0.03);
      \coordinate (fK1_1) at ($(3, 2.5)!0.5!(3.5, 2)$); \fill[bubbleBlue] (fK1_1) circle (0.03);
      \coordinate (fK2_1) at ($(3.5,2)!0.5!(3.5,1)$);   \fill[bubbleBlue] (fK2_1) circle (0.03);
      \coordinate (fK3_1) at ($(2.5,1.2)!0.5!(2,0.2)$); \fill[bubbleBlue] (fK3_1) circle (0.03);
      \coordinate (fK4_1) at ($(1.2,0.8)!0.5!(0.5,1)$); \fill[bubbleBlue] (fK4_1) circle (0.03);
      \coordinate (fK5_1) at ($(0.8,2.2)!0.5!(0.5,3.2)$); \fill[bubbleBlue] (fK5_1) circle (0.03);
      \coordinate (fK6_1) at ($(1.5,3)!0.5!(2.5,3.5)$); \fill[bubbleBlue] (fK6_1) circle (0.03);

      \coordinate (node_K2) at (2.09, 1.97);
      \coordinate (node_K4) at (1.92, 1.97);
      \coordinate (node_K5) at (1.993, 2.1);
      \fill[nodeRed] (node_K2) circle (0.03);
      \fill[nodeRed] (node_K4) circle (0.03);
      \fill[nodeRed] (node_K5) circle (0.03);
    \end{tikzpicture}
     \caption{Schematic representation of degrees of freedom in the DDR scheme with face-bubble enrichment. The discrete space $\UD$ consists of vertex displacements and face-bubble corrections.}
    \label{sec6:Dofs_fig}
  \end{minipage}
  \hfill
  \begin{minipage}[t]{0.48\textwidth}
    \vspace{0pt} 
    \textbf{Degree of freedom configuration:}   
   \smallskip
\begin{itemize}[label=\textbullet]
\item \textbf{Fractures~(red lines):}  The geometry $\Gamma$ represents discontinuities~(fractures) in the domain. 

  \item \textbf{Vertex DOFs~(red dots):} Each vertex $s$ carries several displacement degrees of freedom $\mathbf{v}_s \in \mathbb{R}^d$, one for each connected component of $\Omega\backslash\Gamma$ around it.
  
 \item \textbf{Face-bubble DOFs~(blue dots):} Each face $\sigma$ is enriched with a {bubble} degrees of freedom that corrects the normal component of displacement. On fracture faces~(red lines), two {bubble} DOFs are assigned to capture normal flux on each side, and allow for the reconstruction of the jumps across the fracture. On non-fracture faces, a single {bubble} DOF per face provides {locking-free} enrichment.
  \end{itemize}
  \end{minipage}
\end{figure}

With this in mind, the discrete space, incorporating homogeneous Dirichlet boundary conditions on $\partial\Omega$, becomes
\begin{equation*}
\begin{aligned}\UDz = \Big\{\bv_\D={}&((\bv_{\mathcal{K}s})_{K\in\cells,\,s\in\nodes_K},(v_{K\sigma})_{K\in\cells,\,\sigma\in\mathcal F_{K}})\,:
\bv_{\mathcal{K}s}\in\R^d\,,\; v_{K\sigma}\in\R,\\
&\bv_{\mathcal{K}s}=\bv_{\mathcal{L}s}
\mbox{ if $K,L \in\cells_s$ are on the same side of $\Gamma$},\;\; v_{K\sigma}=v_{L\sigma}\mbox{ if $ \sigma \cap\Gamma=\emptyset$},\\
&\bv_{\mathcal{K}s}=\mathbf{0}\mbox{ if $s\in\nodes^{\text{ext}}$},\, v_{K\sigma}=0 \text{ if }\sigma\in \cF^{\text{ext}}
\Big\}.
	\end{aligned}
\end{equation*}
Our approach differs from that in \cite[Section 3.2]{MR4886366} in the sense that {bubble} DOFs are introduced on every face $\sigma \in \cF$, whereas in \cite{MR4886366} they appear only on the fracture network. The approximation space $C_{\D}$ of the Lagrange multiplier ${\Psi}$
is carried out in the same manner as in \cite{MR4886366}, albeit taking into account the absence of friction, namely, we use the space of piecewise constant scalar functions on $\faces_\Gamma$ (representing the negative of normal surface tractions):
$$M_{\D}= \big\{\Psi_{\D} \in L^2\left(  \Gamma \right)\,:\, \Psi_\sigma :=(\Psi_{\D})_{|\sigma}~\text{is constant for all $\sigma\in\faces_\Gamma$} \big\}.$$ 
 We then introduce the discrete dual cone
 $$C_{\D} = \big\{ \Psi_{\D} \in M_{\D} \,:\, \Psi_\D \ge 0\}.$$ 

The interpolator introduced in~\cite[Subsection~3.4]{MR4886366} is also modified to account for the new discrete space~$\UDz$. 
The space $\mathbf{\mathcal{C}}^0_0(\overline{\Omega}\backslash \Gamma)$ is spanned by functions that are continuous on $\overline{\Omega}\backslash\Gamma$, have limits on each side of~$\Gamma$, and vanish on~$\partial\Omega$. The interpolator $\IUD:\mathbf{\mathcal{C}}^0_0(\overline{\Omega}\backslash\Gamma)\to\UDz$ is defined, for any $\bv\in \mathbf{\mathcal{C}}^0_0(\overline{\Omega}\backslash\Gamma)$, by
\begin{equation*}
  \begin{aligned}
  (\IUD \bv )_{\mathcal{K}s} ={}& \bv_{|K}(\mathbf{x}_s)&&\quad\forall K\in\cells\,,\;\forall s \in \nodes_K,\\
  (\IUD \bv)_{K\sigma} ={}& \frac{1}{|\sigma|} \int_{\sigma} ( \gamma^{K\sigma} \bv - \Pi^{K\sigma}(\IUD\bv))\cdot \normal_{K\sigma} &&\quad\forall K\in\cells\,,\;\forall\sigma \in \faces_{K},
  \end{aligned}
\end{equation*}
where the definitions of $\gamma^{K\sigma}$ and $\Pi^{K\sigma}$ follow from those of $\gamma^{\sigma}$ and $\Pi^{\sigma}$, respectively, for each face $\sigma$ of cell $K$.

Let $\sigma\in\cF_\Gamma$ be a fracture face, and $K$~(resp. $L$) is the cell on the positive~(resp. negative) side of $\Gamma$. We define the normal displacement jump operator on $\sigma$ as $\jump{\cdot}_{\sigma,\normal}: \UDz \to \Poly{0}(\sigma)$ such that, for all $\bv_{\D} \in \UDz$,
$$\jump{\bv_{\D}}_{\sigma,\normal}=\left(\overline{\bv}_{K\sigma}-\overline{\bv}_{L\sigma}\right)\cdot\normal_{K\sigma}+\left(v_{K\sigma}-v_{L\sigma}\right),$$ where the definition of $\overline{\bv}_{K\sigma}$~(resp. $\overline{\bv}_{L\sigma}$) follows from \eqref{eq:average.Pisigma}. 
The local jump operators are then patched together to a global, piecewise polynomial operator $\jump{\cdot}_{\D,\normal}: \UDz\to \Poly{0} ( \cF_\Gamma)$, where $\Poly{0}(\cF_\Gamma)$ denotes the space of piecewise constants on $\cF_\Gamma$. This operator satisfies: for all $\bv_{\D}\in \UDz$,
$$\left(\jump{\bv_{\D}}_{\D,\normal}\right)|_{\sigma}=\jump{\bv_{\D}}_{\sigma,\normal}.$$

The numerical scheme corresponding to the mixed variational formulation of the fracture problem is: Find $( \bu_{\D}, {\Psi}_{\D} ) \in \UDz \times C_{\D}$
such that, for all $( \bv_{\D}, {\mu}_{\D} ) \in \UDz \times C_{\D}$,
\begin{subequations}
\begin{align}
& \int_{\Omega} \bbsigma_{\D} (\bu_{\D}): \bbeps_{\D}\left(\bv_{\D} \right)  + \mu_{1} S_{\D}\left( \bu_{\D}, \bv_{\D} \right) + \int_{\Gamma} {\Psi}_{\D}\jump{\bv_{\D}}_{\D,\normal} = \int_\Omega  \mathbf{f} \cdot \wt \Pi^{\D} \bv_{\D}, \label{sec6:eq:meca.var.tresca}  \\[2.5ex]
& \int_{\Gamma} \left( {\mu}_{\D} - {\Psi}_{\D} \right) \jump{\bu_{\D}}_{\D,\normal}  \le  0, \label{sec6:eq:meca.ineq.contact.tresca}
\end{align}
\end{subequations}

The analysis done in Section \ref{sec:proofs} can be combined with the one of \cite{DKMR25} to show that this scheme, combining bubble functions to handle the fractures and to deliver a locking-free method, satisfies the following error estimates.

\begin{theorem}[Abstract error estimate]\label{Sec5:abs_errbd_thm51} Recall the definition of the norm $\NORM{e,\D}{\cdot}$ induced by the inner product defined in \eqref{discrete:energy:ps}. For $(\bu_{\D}, {\Psi}_{\D})$ solution of \eqref{sec6:eq:meca.var.tresca}-\eqref{sec6:eq:meca.ineq.contact.tresca} and $(\bu,{\Psi})$ solution of  mixed variational formulation of \eqref{sec6:mode_tresca_1}, we have the following abstract error estimate:
\begin{align*}
\NORM{e,\D}{\bu_{\D}-\IUD\bu}\lesssim{}& \frac{1}{\sqrt{\mu_{1}}} \mathcal{W}_{\D}(\bbsig(\bu))+ \sqrt{\mu_{2}}C_\D(\bu, \IUD\bu)\notag\\
  &+\left( 
\NORM{L^{2}(\Gamma)}{{\Psi} - \pi^0_{\cF_\Gamma}{\Psi}}\NORM{L^{2}(\Gamma)}{\jump{\IUD\bu}_{\D} - \jump{\bu}}\right)^{1/2},\\
\|{\Psi}_{\D}-{\Psi}\|_{-1/2,\Gamma}\lesssim{}&  \mathcal{W}_{\D}(\bbsig(\bu))+\mu_{2}C_{\D}(\bu, \IUD\bu) +\|\pi^0_{\cF_\Gamma}{\Psi}-{\Psi}\|_{-1/2,\Gamma}\\& +\left(\sqrt{\mu_{2}}+\sqrt{\la}\right)\NORM{e,\D}{\bu_{\D}-\IUD\bu},
\end{align*}
where $\pi^0_{\cF_\Gamma}$ is the orthogonal projection on $\mathbb{P}^0(\faces_\Gamma)$, and the norm $\NORM{-\nicefrac12,\Gamma}{{\cdot}}$ on ${L}^2(\Gamma)$ is defined by: for all ${\Psi} \in {L}^2(\Gamma)$,
 \begin{equation*}
 \NORM{-\nicefrac12,\Gamma}{{\Psi}}
=\sup_{\bv\in \mathbf{H}^1_0(\Omega\setminus\Gamma)\backslash\{\mathbf{0}\}}
\frac{\int_{\Gamma}{\Psi}\jump{\bv}_\normal}{\NORM{\mathbf{H}^1(\Omega\setminus\Gamma)}{\bv}}.
\end{equation*}
\end{theorem}

In the following theorem, we denote by ${H}^1(\mathcal{F}_\Gamma)$
the space of functions defined on $\Gamma$ and ${H}^1$ on each $\sigma \in \mathcal{F}_\Gamma$. 

\begin{theorem}[Error estimate]\label{sec6:th:error_estimate}
Let $(\bu,\Psi)$ be the solution to the mixed formulation and assume that 
$\bu \in \mathbf{H}^2(\mathcal{M})$ and $\Psi \in {H}^1(\mathcal{F}_\Gamma)$. 
Then the solution $(\bu_{\D},{\Psi}_{\D})$ of \eqref{sec6:eq:meca.var.tresca}-\eqref{sec6:eq:meca.ineq.contact.tresca}
satisfies the following error estimate:
\begin{align}
\NORM{e,\D}{\bu_{\D}-\IUD\bu}\lesssim{}&
h \left(\frac{\mu_2}{\sqrt{\mu_{1}}}\SEMINORM{H^2(\cells)}{\bu}+\frac{\la}{\sqrt{\mu_{1}}}\SEMINORM{H^1(\cells)}{\divScal\bu} +|{\Psi}|_{H^1(\mathcal{F}_\Gamma)}+ |\jump{\bu}|_{H^1(\mathcal{F}_\Gamma)} \right),\label{thm3:err_est_1}\\
\|{\Psi}_{\D} - {\Psi}\|_{-1/2,\Gamma}\lesssim{}& h \bigg((1+\mu_2)\SEMINORM{H^2(\cells)}{\bu}+\la\SEMINORM{H^1(\cells)}{\divScal\bu} \nonumber+|{\Psi}|_{H^1(\mathcal{F}_\Gamma)}\bigg)\\
&
 +\left(\sqrt{\mu_{2}}+\sqrt{\la}\right)\NORM{e,\D}{\bu_{\D}-\IUD\bu},\label{thm4:err_est_1}
\end{align}
where the coefficient hidden in $\lsim$ is independent of the Lam\'e coefficients $\mu$ and $\la$.  
\end{theorem}

\begin{remark}[Locking-free estimate]
The regularity estimate \eqref{reg1} was established for pure elasticity on convex domains. For the contact problem, the domain $\Omega \setminus \Gamma$ is non-convex due to the fracture network, and this regularity result then seems out of reach. While we cannot directly apply \eqref{reg1} to guarantee a uniform bound (in terms of $\la$) of the quantity in the right-hand side of the estimate in~\eqref{thm3:err_est_1}, we emphasize that the error estimate itself has a level or robustness for the displacement~($\bu$) in the sense that the coefficient hidden in $\lesssim$ does not explicitly depend on the Lamé coefficient $\la$. In contrast, the estimate for Lagrange multiplier~$(\Psi)$ in \eqref{thm4:err_est_1} deteriorates at rate $\sqrt{\la}$ (which is expected since this Lagrange multiplier represents a normal surface traction that directly depends on the Lamé coefficients).
\end{remark}
\begin{remark}[Extension to contact mechanics with Tresca friction]\label{rem:tresca}
For the Tresca friction model~\cite[Equation~2]{MR4886366}, we need to consider {vector bubble} enrichments of the discrete displacement space (rather than scalar {bubbles}), a distinction necessary to satisfy the inf--sup condition~\cite[Theorem 5.1]{MR4886366}. With this enrichment in place, the analysis in \cite{MR4886366} can be adapted to show that the error bounds in Theorems~\ref{Sec5:abs_errbd_thm51}-\ref{sec6:th:error_estimate} also hold for the contact problem with Tresca friction.
\end{remark}

\section{Numerical Experiments} \label{sec:numerics}
This section presents numerical experiments that validate the theoretical convergence rates and robustness of the proposed numerical scheme \eqref{mixed_discrete} across various parameter regimes. We verify the {locking-free} property by testing with varying Lam\'e coefficients $\la$ ranging from $\la = 1$ to $\la = 10^8$ and $\mu=1$, demonstrating that the numerical method maintains uniform convergence independently of the incompressibility parameter. Two test cases with manufactured solutions are considered on both Voronoi and general tetrahedral meshes,  see Fig.~\ref{Exm2_fig1}.

\subsection{Convergence rate definition}
For all numerical experiments, we measure convergence using the relative error in the $L^2$-norm of the symmetric gradient:
\begin{equation*}
E_{\text{Relative}} = \frac{\NORM{L^2(\Omega)}{\bbeps_{\D}(\bu_\D)- \bbeps_{\D} (\IUD \bu) }}{\NORM{L^2(\Omega)}{\bbeps(\bu)}},
\end{equation*}
where $\bu$ is the exact manufactured solution and $\bu_\D$ is the numerical solution calculated by the proposed discrete scheme~\eqref{mixed_discrete}. The convergence rate is determined by comparing errors across successive mesh refinements via:
\begin{equation*}
\text{Rate} = \frac{\log(E_{\text{Relative}}^{\text{i}} / E_{\text{Relative}}^{\text{i+1}})}{\log(h_{\text{i}} / h_{\text{i+1}})}.
\end{equation*}
According to Theorem \ref{error_estimate}, the expected convergence rate is $\mathcal{O}(h)$. 
\begin{example}\label{exm_1}  
In this example, we consider the domain $\Omega=(0,1)^{3}$. The manufactured solution $\bu$ is taken as 
\begin{equation*}
  \bu =\begin{pmatrix} -2\,\sin(\pi x)\cos(\pi y)\cos(\pi z) \\[0.5em]
  \sin(\pi y)\cos(\pi x)\cos(\pi z) \\[0.5em]
  \sin(\pi z)\cos(\pi x)\cos(\pi y)
  \end{pmatrix}.
\end{equation*}
This choice is particularly challenging, as it examines the method in the nearly incompressible limit $\la \to+\infty$, where standard nodal methods are prone to volumetric locking.  The exact displacement field satisfies $\divScal\, \bu=0$ and the force on the right side $\mathbf{f}$ is calculated to ensure consistency with the exact solution.

Table~\ref{exm1:tab_voro} presents convergence results on Voronoi meshes with varying mesh sizes and Lam\'e coefficients. The numerical data demonstrates the robustness and accuracy of the proposed method. The results show that not only the convergence rates but, actually, the magnitudes of the errors are virtually identical for all tested values of $\la$, which demonstrates the full robustness with respect to the incompressibility parameter and confirms the {locking-free} nature of the formulation. The observed rates consistently match the theoretical prediction of first-order convergence from Theorem~\ref{error_estimate}, with several instances exhibiting mild {superconvergence}. 

\begin{table}[h]
\centering
\caption{Convergence analysis results for Example \ref{exm_1} on Voronoi meshes.}
\label{exm1:tab_voro}
\scriptsize
\begin{tabular}{c|cc|cc|cc}
\hline
\text{MeshSize} & \multicolumn{2}{c|}{$\la = 1$} & \multicolumn{2}{c|}{$\la = 10^3$} & \multicolumn{2}{c}{$\la = 10^6$} \\
\cline{2-7}
& $E_{\text{Relative}}$ & Rate & $E_{\text{Relative}}$ & Rate & $E_{\text{Relative}}$ & Rate \\
\hline
3.053127e-01 & 1.575633e-01 & --- & 1.571255e-01 & --- & 1.571270e-01 & --- \\
2.213817e-01 & 1.135921e-01 & 1.0179 & 1.131257e-01 & 1.0221 & 1.131257e-01 & 1.0221 \\
1.767538e-01 & 8.283524e-02 & 1.4026 & 8.230529e-02 & 1.4128 & 8.230508e-02 & 1.4128 \\
1.500044e-01 & 6.334517e-02 & 1.6348 & 6.288164e-02 & 1.6404 & 6.288141e-02 & 1.6404 \\
1.289197e-01 & 5.151087e-02 & 1.3653 & 5.106443e-02 & 1.3743 & 5.106412e-02 & 1.3743 \\
\hline
\end{tabular}
\end{table}
Table~\ref{exm1:tab_tetcube} presents the convergence results on general tetrahedral meshes. Compared to Voronoi meshes, the tetrahedral meshes exhibit larger errors on coarse grids; however, convergence is unaffected. After the first few refinements, the method achieves stable convergence rates close to 1.0, confirming the expected asymptotic optimality for sufficiently fine meshes. As for Voronoi meshes, the convergence rates and error magnitudes remain virtually independent of the Lam\'e parameter $\la$, demonstrating that the {locking-free} property extends consistently to general tetrahedral discretizations. 
\begin{table}[h]
\centering
\caption{Convergence analysis results for Example \ref{exm_1} on general tetrahedral.}
\label{exm1:tab_tetcube}
\scriptsize
\begin{tabular}{c|cc|cc|cc}
\hline
\text{MeshSize} & \multicolumn{2}{c|}{$\la = 1$} & \multicolumn{2}{c|}{$\la = 10^3$} & \multicolumn{2}{c}{$\la = 10^6$} \\
\cline{2-7}
& $E_{\text{Relative}}$ & Rate & $E_{\text{Relative}}$ & Rate & $E_{\text{Relative}}$ & Rate \\
\hline
5.589426e-01 & 5.087922e-01 & --- & 5.055304e-01 & --- & 5.055289e-01 & --- \\
4.998278e-01 & 4.676925e-01 & 0.7535 & 4.629756e-01 & 0.7866 & 4.629747e-01 & 0.7866 \\
3.920304e-01 & 3.947911e-01 & 0.6976 & 3.917211e-01 & 0.6880 & 3.917200e-01 & 0.6880 \\
3.130676e-01 & 3.130434e-01 & 1.0315 & 3.110712e-01 & 1.0249 & 3.110703e-01 & 1.0249 \\
2.567587e-01 & 2.600056e-01 & 0.9362 & 2.584356e-01 & 0.9349 & 2.584354e-01 & 0.9349 \\
\hline
\end{tabular}
\end{table}
\end{example}

\begin{example}\label{exm2}
We consider the cubic domain $\Omega = (0,1)^3$ and validate the method using a manufactured solution with Lam\'e-parameter dependence. The displacement field $\bu$ is defined as
\begin{equation*}
\bu =\begin{pmatrix}
-2\sin(2\pi x)\cos(2\pi y)\cos(2\pi z) \\[0.5em]
\sin(2\pi y)\cos(2\pi x)\cos(2\pi z) \\[0.5em]
\sin(2\pi z)\cos(2\pi x)\cos(2\pi y)
\end{pmatrix}
+
\frac{1}{\la}
\begin{pmatrix}
\sin(2\pi x) \\[0.5em]
\sin(2\pi y) \\[0.5em]
\sin(2\pi z)
\end{pmatrix}.
\end{equation*}
The difference with the exact solution in Example \ref{exm_1} is that, here, we do not have $\divScal\bu=0$. The magnitude of the discrete displacement $\bu_{\D}$ on Voronoi and general tetrahedral meshes with $\la=1$ is shown in Fig.~\ref{Exm2_fig1}. This choice introduces parameter-dependent behavior, providing a more rigorous test of the method's robustness in the quasi-incompressible limit. We, however, note that the exact solution is designed so that $\la\divScal\bu$~(and the corresponding source term $\mathbf{f}$) remains bounded as $\la\to+\infty$, which allows us to expect still error estimates that do degrade in the quasi-incompressible limit.
Tables~\ref{exm2:tab_voro} and~\ref{exm2:tab_tetcube} summarize the convergence results on Voronoi and general tetrahedral meshes, respectively, with $\la$ ranging from $1$ to $10^{8}$. 

The numerical results demonstrate mild {superconvergence} on Voronoi and optimal first-order convergence on general tetrahedral meshes across all tested $\la$. The errors show a mild dependency on $\la$, but their magnitude remains very little impacted by this coefficient, including in the quasi-incompressible limit. This confirms again the {locking-free} nature of the formulation.

Table~\ref{exm2:tab_voro_without_bubble} shows the convergence behavior of the nodal scheme without {bubble} enrichment, underscoring the crucial role of {bubble} degrees of freedom. As $\la$ increases, accuracy deteriorates severely -- errors surge to $10^5$ for $\la=10^{6}$ and $10^7$ for $\la=10^{8}$ -- clearly revealing {volumetric locking}. In contrast, the proposed method~(Table~\ref{exm2:tab_voro}) maintains $\la$--independent accuracy. These results demonstrate that scalar face-bubble unknowns are essential to ensure {locking-free} robustness by providing sufficient flexibility in the divergence field, without which the method fails in the quasi-incompressible regime.

\begin{figure}[H]
  		\centering
  		\includegraphics[width=6.00cm, height=4.20cm]{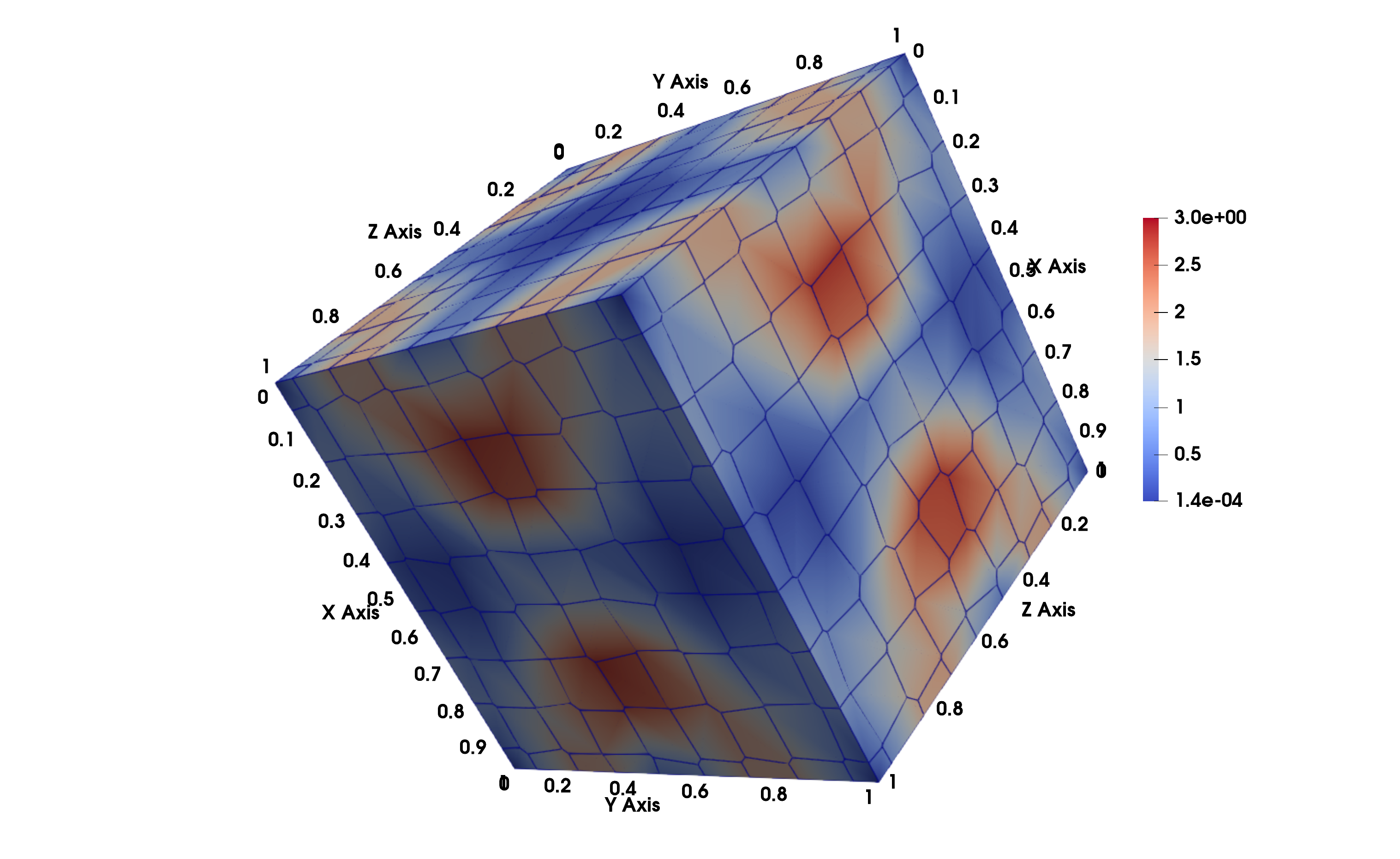}
  		\includegraphics[width=6.00cm, height=4.20cm]{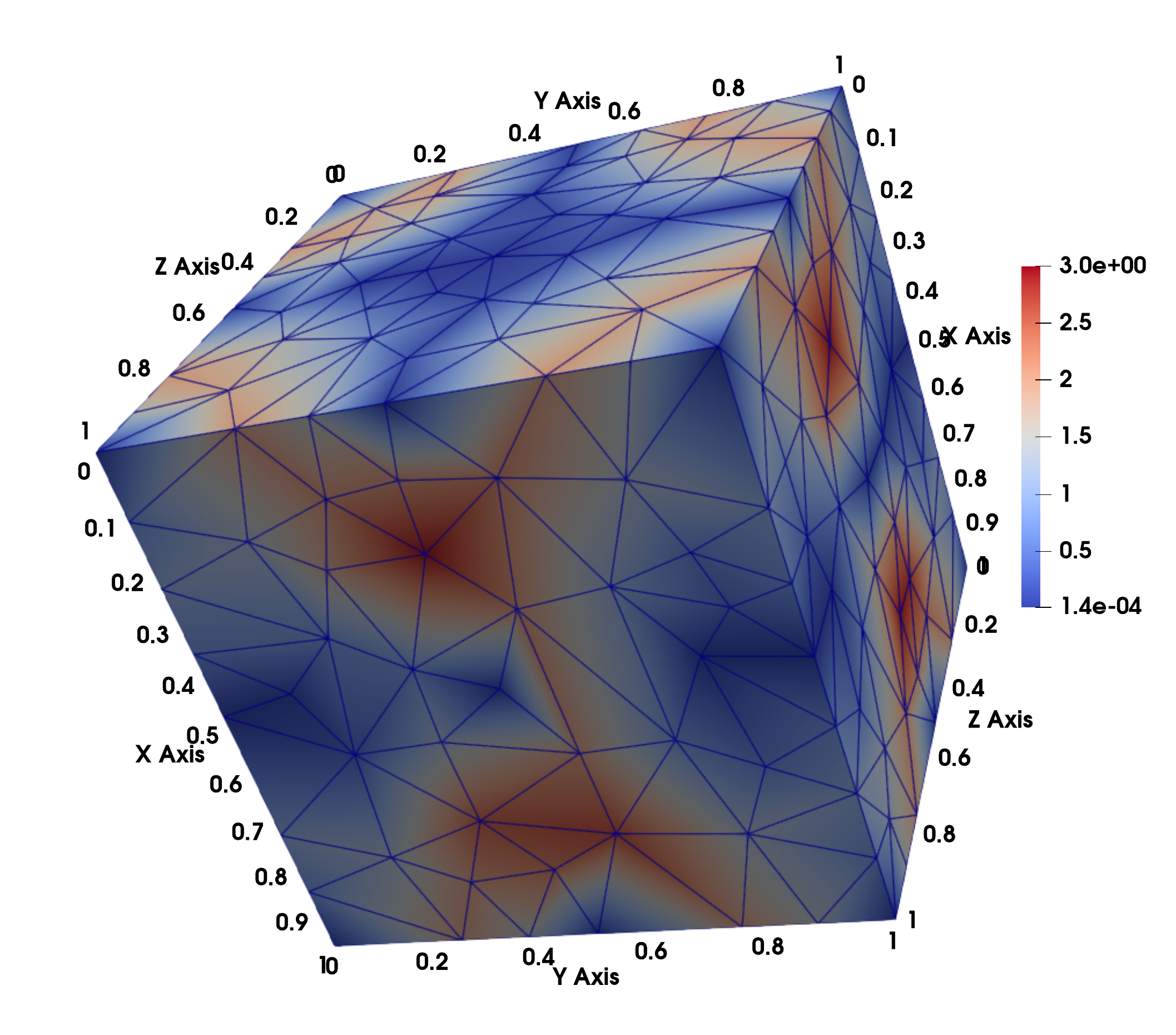}
  	\caption{(Example~\ref{exm2}). Discrete displacement magnitude~$\bu_\D$ on the Voronoi~(left) and tetrahedral~(right) meshes with mesh sizes $h = 0.3130676$ and $h = 0.2213817$, respectively.}
  	\label{Exm2_fig1}
  \end{figure}
\begin{table}[H]
\centering
\caption{Convergence analysis results for Example \ref{exm2} on Voronoi meshes.}
\label{exm2:tab_voro}
\scriptsize
\begin{tabular}{c|cc|cc|cc|cc}
\hline
\text{MeshSize} & \multicolumn{2}{c|}{$\lambda = 1$} & \multicolumn{2}{c|}{$\lambda = 10^3$} & \multicolumn{2}{c|}{$\lambda = 10^6$} & \multicolumn{2}{c}{$\lambda = 10^8$} \\
\cline{2-9}
& $E_{\text{Relative}}$ & Rate & $E_{\text{Relative}}$ & Rate & $E_{\text{Relative}}$ & Rate & $E_{\text{Relative}}$ & Rate \\
\hline
3.053127e-01 & 3.299394e-01 & --- & 4.123398e-01 & --- & 4.123217e-01 & --- & 4.123217e-01 & --- \\
2.213817e-01 & 2.705299e-01 & 0.6176 & 3.549783e-01 & 0.4660 & 3.549652e-01 & 0.4660 & 3.549652e-01 & 0.4660 \\
1.767538e-01 & 2.120576e-01 & 1.0817 & 2.796525e-01 & 1.0594 & 2.796438e-01 & 1.0594 & 2.796438e-01 & 1.0594 \\
1.500044e-01 & 1.704391e-01 & 1.3314 & 2.248553e-01 & 1.3291 & 2.248480e-01 & 1.3291 & 2.248480e-01 & 1.3291 \\
1.289197e-01 & 1.381700e-01 & 1.3857 & 1.812788e-01 & 1.4222 & 1.812725e-01 & 1.4222 & 1.812725e-01 & 1.4222 \\
\hline
\end{tabular}
\end{table}
\begin{table}[H]
\centering
\caption{Convergence analysis results for Example \ref{exm2} on tetrahedral meshes.}
\label{exm2:tab_tetcube}
\scriptsize
\begin{tabular}{c|cc|cc|cc|cc}
\hline
\text{MeshSize} & \multicolumn{2}{c|}{$\lambda = 1$} & \multicolumn{2}{c|}{$\lambda = 10^3$} & \multicolumn{2}{c|}{$\lambda = 10^6$} & \multicolumn{2}{c}{$\lambda = 10^8$} \\
\cline{2-9}
& $E_{\text{Relative}}$ & Rate & $E_{\text{Relative}}$ & Rate & $E_{\text{Relative}}$ & Rate & $E_{\text{Relative}}$ & Rate \\
\hline
5.589426e-01 & 7.531295e-01 & --- & 8.272468e-01 & --- & 8.271323e-01 & --- & 8.271321e-01 & --- \\
4.998278e-01 & 6.431486e-01 & 1.4122 & 7.104305e-01 & 1.3619 & 7.103314e-01 & 1.3619 & 7.103313e-01 & 1.3619 \\
3.920304e-01 & 5.774748e-01 & 0.4434 & 6.613396e-01 & 0.2948 & 6.612501e-01 & 0.2947 & 6.612500e-01 & 0.2947 \\
3.130676e-01 & 4.826357e-01 & 0.7976 & 5.731405e-01 & 0.6364 & 5.730768e-01 & 0.6363 & 5.730767e-01 & 0.6363 \\
2.567587e-01 & 4.035182e-01 & 0.9030 & 4.856101e-01 & 0.8358 & 4.855605e-01 & 0.8358 & 4.855605e-01 & 0.8358 \\
\hline
\end{tabular}
\end{table}  
\end{example}

\begin{table}[H]
\centering
\caption{Convergence analysis results for Example \ref{exm2} on Voronoi meshes.}
\label{exm2:tab_voro_without_bubble}
\scriptsize
\begin{tabular}{c|cc|c|c|c}
\hline
\text{MeshSize} & \multicolumn{2}{c|}{$\lambda = 1$} & $\lambda = 10^3$ & $\lambda = 10^6$ & $\lambda = 10^8$\\
\cline{2-6}
& $E_{\text{Relative}}$ & Rate & $E_{\text{Relative}}$ & $E_{\text{Relative}}$ &  $E_{\text{Relative}}$  \\
\hline
3.053e-01 & 1.5939e+00 & --- & 3.4186e+02 & 3.4158e+05 & 3.4158e+07 \\
2.214e-01 & 1.2082e+00 & 0.86  & 2.7847e+02 & 2.7834e+05 & 2.7834e+07 \\
1.768e-01 & 1.0430e+00 & 0.65  & 2.4181e+02 & 2.4171e+05 & 2.4171e+07 \\
1.500e-01 & 8.9976e-01 & 0.90  & 2.1101e+02 & 2.1092e+05 & 2.1092e+07 \\
1.289e-01 & 8.0420e-01 & 0.74  & 1.9002e+02 & 1.8994e+05 & 1.8994e+07 \\
\hline
\end{tabular}
\end{table}

\begin{example}\label{exm3}
This example extends the proposed methodology to problems involving mixed boundary conditions. Dirichlet conditions are prescribed on $\Gamma = \{ (x,y,z) \in \partial\Omega \mid x = 0 \}$, while Neumann conditions, $\mathbf{\sigma}(\mathbf{u})\mathbf{n} = {\bf 0}$, are imposed on $\partial\Omega \setminus \Gamma$. The manufactured solution is identical to that of Example~\ref{exm2}, defined on the unit cube $\Omega = (0,1)^3$. The corresponding linear elasticity problem reads
\begin{equation*}
\left\{\!\!\!\!
\begin{array}{lll}
& -\div \bbsigma(\bu)= \mathbf{f}  & \mbox{ on } \Omega,\\[1ex]
& \bbsigma(\bu)\normal={\bf 0} & \mbox{ on }  \partial\Omega\setminus \Gamma,\\[1ex]
&\bu={\bf 0} & \mbox{ on } \Gamma.
\end{array}
\right.
\end{equation*}
Table~\ref{exm3:tab_voro} presents convergence results on Voronoi meshes with $\la$ ranging from $1$ to $10^{8}$. The results confirm that the method extends naturally to mixed boundary value problems without modification. Convergence rates remain robust and independent of the incompressibility parameter, achieving optimal first-order rates after sufficient mesh refinement. The stability and accuracy across the full range of $\la$ values substantiate the method's robustness for practical applications involving nearly incompressible elastic materials with diverse boundary conditions.
\begin{table}[H]
\centering
\caption{Convergence analysis results for Example~\ref{exm3} on Voronoi meshes.}
\label{exm3:tab_voro}
\scriptsize
\begin{tabular}{c|cc|cc|cc|cc}
\hline
\text{MeshSize} & \multicolumn{2}{c|}{$\lambda = 1$} & \multicolumn{2}{c|}{$\lambda = 10^3$} & \multicolumn{2}{c|}{$\lambda = 10^6$} & \multicolumn{2}{c}{$\lambda = 10^8$} \\
\cline{2-9}
& $E_{\text{Relative}}$ & Rate & $E_{\text{Relative}}$ & Rate & $E_{\text{Relative}}$ & Rate & $E_{\text{Relative}}$ & Rate \\
\hline
3.053127e-01 & 4.029853e-01 & --- & 5.304031e-01 & --- & 5.303617e-01 & --- & 5.303617e-01 & --- \\
2.213817e-01 & 3.370172e-01 & 0.5561 & 4.680815e-01 & 0.3888 & 4.680581e-01 & 0.3888 & 4.680580e-01 & 0.3888 \\
1.767538e-01 & 2.748554e-01 & 0.9056 & 3.775619e-01 & 0.9546 & 3.775393e-01 & 0.9546 & 3.775393e-01 & 0.9546 \\
1.500044e-01 & 2.258388e-01 & 1.1970 & 3.077235e-01 & 1.2464 & 3.077034e-01 & 1.2465 & 3.077034e-01 & 1.2465 \\
1.289197e-01 & 1.868435e-01 & 1.2514 & 2.502349e-01 & 1.3653 & 2.502159e-01 & 1.3653 & 2.502159e-01 & 1.3653 \\
\hline
\end{tabular}
\end{table}
\end{example}

\begin{remark}[Superconvergence on Voronoi meshes]
Mild {superconvergence} is observed in Tables~\ref{exm1:tab_voro}, \ref{exm2:tab_voro}, and \ref{exm3:tab_voro} on Voronoi meshes, where the numerical solution exhibits a convergence rate slightly higher than the theoretically first-order estimate. Compared with tetrahedral meshes, Voronoi meshes contain a larger number of faces. Since the proposed scheme enriches the discrete space by introducing one additional scalar (normal) degree of freedom per face -- acting as a higher-order correction to the underlying low-order formulation -- this enrichment across differently oriented faces incorporates extra geometric information into the approximation. Consequently, the method attains greater flexibility and is able to capture higher-order components of the solution more accurately, leading to enhanced convergence. This effect is less pronounced on tetrahedral meshes, which, for a given mesh size, possess fewer faces, thus yielding only the optimal convergence rate.
\end{remark}

\section{Conclusions}
This work provides a robust and flexible numerical framework for elastic problems that successfully eliminates volumetric locking while maintaining ease of implementation and computational efficiency. The theoretical foundation through optimal error estimates, combined with convincing numerical evidence, establishes the proposed method as a reliable tool for engineering analyses involving nearly incompressible materials. The polytopal nature of the approach positions it well for future extensions to more complex problem classes and heterogeneous media.

\section*{Acknowledgement}
Funded by the European Union (ERC Synergy, NEMESIS, project number 101115663). Views and opinions expressed are, however, those of the authors only and do not necessarily reflect those of the European Union or the European Research Council Executive Agency. Neither the European Union nor the granting authority can be held responsible for them.
%
\printbibliography
\end{document}